\documentclass[11pt]{article}
\usepackage[pagebackref,colorlinks=true,pdfpagemode=none,urlcolor=blue,linkcolor=blue,citecolor=blue]{hyperref}

\usepackage{amsthm, amssymb, bm, bbm}
\usepackage{soul, color,cancel}
\usepackage[]{amsmath}
\usepackage[]{amsfonts}
\usepackage[]{fancyhdr}
\usepackage[]{graphicx}
\graphicspath{{/EPSF/}{../figures/}{./figures/}}

\newtheorem{theorem}{Theorem}[]
\newtheorem{lemma}[theorem]{Lemma}
\newtheorem{proposition}[theorem]{Proposition}

\newtheorem{remark}[theorem]{Remark}

\def \Dm {M}

\def \Rm {\mathbb{R}}
\def \Sm {\mathbb{S}}

\def \Zm {\mathbb{Z}}

\def\B{\mathcal{B}}
\def\C{\mathcal{C}}
\def\D{\mathcal{D}}

\def\SS{\mathcal{S}}
\def\R{\mathcal{R}}
\def\V{\mathcal{V}}

\newcommand{\bd}{\mathbf d}

\newcommand{\bff}{ {\bf f}}

\newcommand{\cout}[1]{}

\newcommand{\x}{\mathrm{x}}

\newcommand{\btheta}{\boldsymbol \theta}

\newcommand{\sgn}[1]{\,{\rm sign}(#1)}

\newcommand{\dprod}[2]{\langle{#1},{#2}\rangle}
\newcommand{\dbar}{\overline{\partial}}
\newcommand{\zbar}{\overline{z}}

\newcommand{\tspan}{\text{span}\ }

\newcommand{\arXiv}[1]{ {\tt \href{http://arxiv.org/abs/#1}{arXiv:#1}} }

\makeatletter
\newdimen\bibindent
\renewenvironment{thebibliography}[1]
  {\par\bigskip\footnotesize
   \section*{\centering\normalfont\footnotesize\MakeUppercase{\refname}}
   \@mkboth{\MakeUppercase{\refname}}{\MakeUppercase{\refname}}
   \list{\@biblabel{\arabic{enumi}}}%
        {\settowidth\labelwidth{\@biblabel{#1}}%
         \leftmargin\labelwidth
         \advance\leftmargin\labelsep
         \advance\leftmargin\bibindent
         \itemindent -\bibindent
         \listparindent \itemindent
         \parsep \z@
         \usecounter{enumi}%
         \let\p@enumi\@empty
         }%
     \renewcommand\newblock{\hskip .11em \@plus.33em \@minus.07em}%
     \sloppy\clubpenalty4000\widowpenalty4000%
     \frenchspacing\footnotesize}
  {\def\@noitemerr
     {\@latex@warning{Empty `thebibliography' environment}}%
   \endlist}
\makeatother

\title{Efficient tensor tomography in fan-beam coordinates}
\author{Fran\c{c}ois Monard \thanks{Department of Mathematics, University of Michigan, 2074 East Hall, 530 Church Street, Ann Arbor, MI 48109-1043. Email: monard@umich.edu }}
\date{}

\begin{document}
\maketitle

\begin{abstract}
    We propose a thorough analysis of the tensor tomography problem on the Euclidean unit disk parameterized in fan-beam coordinates. This includes, for the inversion of the Radon transform over functions, using another range characterization first appearing in \cite{Pestov2004} to enforce in a fast way classical moment conditions at all orders. When considering direction-dependent integrands (e.g., tensors), a problem where injectivity no longer holds, we propose a suitable representative (other than the traditionally sought-after solenoidal candidate) to be reconstructed, as well as an efficient procedure to do so. Numerical examples illustrating the method are provided at the end.
\end{abstract}

\section{Introduction}

The present work proposes a detailed account of the {\em tensor tomography problem} (TTP) on the Euclidean unit disk in fan-beam coordinates, including inversion procedures (modulo kernel) and full range description of the operator over tensors of arbitrary order. 

The tensor tomography problem, posed on a non-trapping Riemannian surface $(M,g)$, consists of (i) determining what part of a symmetric $m$-covariant tensor $f$ ($m\ge 0$) is reconstructible from knowledge of its geodesic X-ray transform (XRT) $If$ (the collection of its integrals along all geodesics passing through that surface), and (ii) how to reconstruct a faithful representative of $f$ from $If$ modulo the kernel of $I$. Such a problem arises for its applications to imaging sciences as well as for its ties to integral geometric problems such as spectral and boundary rigidity, and Calder\'on's inverse conductivity problem. Particular applications are Computerized Tomography ($m=0$), Ultrasound Doppler Tomography ($m=1$, see \cite{Anikonov1997}), deformation boundary rigidity ($m=0,2$, see \cite{Sharafudtinov2007}) and imaging of elastic media with slightly anisotropic properties ($m=4$, see \cite{Sharafudtinov1994}). 

Regarding question (i), for any $m\ge 1$, this problem is famously non-injective, as the {\em inner derivative} operator $\bd$ (see \cite{Sharafudtinov1994} for details) generates an element in the kernel of the XRT out of any symmetric tensor vanishing at the boundary of $M$. In this regard, a way of reformulating question (i) is to ask whether these elements are the only ones in the kernel, a question which received positive answers in the case of surfaces of revolution satisfying Herglotz' condition in \cite{Sharafudtinov1997} (in fact, all dimensions $\ge 2$ there), and in the case of {\em simple\footnote{A non-trapping Riemannian manifold (dim $\ge 2$) with boundary is called {\em simple} if it has no conjugate points and its boundary is strictly convex.} surfaces} in \cite{Paternain2011a}. Neither case is a subset of the other though both cover the Euclidean disk, studied in the present article. Let us mention that further dynamical tools have recently allowed results on cases with trapped geometry \cite{Guillarmou2014,Guillarmou2014a,Guillarmou2015}, Fourier analysis and representation theory have led to progress on Radon transforms on compact Lie groups and finite groups \cite{Ilmavirt2015,Ilmavirt2014,Ilmavirt2014a}, and microlocal analytic approaches have led to results on general manifolds with conjugate points \cite{Monard2013b,Stefanov2012a,Stefanov2014,Stefanov2013,Uhlmanna,Holman2015}. 

On studying question (ii), the solenoidal-potential decomposition of an $m$ tensor $f = \bd h + f^s$ (with $h$ an $m-1$ tensor vanishing on $\partial M$), true for any $L^2$ tensor $f$ and where each summand is continuous in terms of $f$, suggests that, since $If = If^s$, the ``solenoidal'' tensor $f^s$ is a good candidate as a representative to be reconstructed from data $If$. 
In this direction, Sharafutdinov provided a reconstruction of $f^s$ in Euclidean free space in any dimension in \cite[Theorem 2.12.2]{Sharafudtinov1994}, Kazantsev and Bukhgeim proposed in \cite{Kazantsev2004} a reconstruction algorithm for a tensor of arbitrary order in the case of the Euclidean unit disk, see also \cite{Kazantsev2006,Svetov2013,Derevtsov2005,Derevtsov2011} for further works on the topic and the recent work \cite{Derevtsov2015}. Another approach is to make use of the theory of A-analytic functions {\it \`a la} Bukhgeim, which has successfully led to range characterizations and inversion formulas, most recently in \cite{Sadiq2015,Sadiq2013,Sadiq2014} for the ray transform in Euclidean convex domains over functions, vector fields and two-tensors, including attenuation coefficients. 

The solenoidal representative is arguably not an easy quantity to work with: Sharafutdinov's formulas involve iterated use of the non-local operator $\Delta^{-1} \bd^2$, where $\Delta$ denotes componentwise Laplacian, and the expression of $\bd^2$ depends on the tensor order. A first salient feature of this paper is to exploit another decomposition, arising naturally for instance when performing Fourier analysis on tangent fibers: doing this comes with a different inner product structure than the one traditionally used on tensor fields, and in particular, this motivates another tensor decomposition than the solenoidal-potential one. This decomposition has proved to be more natural in earlier contexts such as the search for conformal Killing tensor fields (see for instance \cite{Sharafudtinov2007,Sharafutdinov2014} or \cite[Theorem 1.5]{Dairbekov2011}). Here we exploit this other decomposition to provide, for any tensor field, an element to be reconstructed with some advantages: its Fourier components are made up of a function to be inverted for via ``traditional'' inverse X-ray transform, plus additional terms which can be easily expressed in terms of analytic functions, the reconstruction of which via {\it ad hoc} Cauchy integrals is straighforward and efficient; in addition, the X-ray transforms of each component live on $L^2$-orthogonal subpaces of data space, and the reconstruction formulas given below encode projection onto each subspace before reconstruction, without requiring intermediate processing. In the case of tensors of odd order, special care is given to the reconstruction of solenoidal one-forms supported up to the boundary, for which inversions in \cite{Pestov2004} only covered the case of compactly supported ones. Some numerical reconstructions are presented at the end of the article. We also briefly discuss the fact that the decomposition presented requires choosing a ``central frequency'', the choice of which may also yield other decompositions whose reconstruction would be equally fast to implement. 

As a second feature of this article, understanding the behavior of data space under this decomposition has also shed some light on the range characterization of the XRT over functions (call it $I_0$ for the case $m=0$) in fan-beam coordinates. More specifically, we prove an equivalence between a range characterization of $I_0$ given by Pestov and Uhlmann in \cite{Pestov2004} on simple Riemannian surfaces, and the classical moment conditions due to Gelfand and Graev \cite{GelfandGraev1960} and Helgason and Ludwig \cite{Helgason1999,Ludwig1966}, translated into fan-beam coordinates. The latter conditions can also be regarded as countably many algebraic conditions characterizing the range of $I_0$ over functions in the so-called parallel geometry. Translating these conditions into fan-beam coordinates has been and continues to be an object of active study \cite{Chen2005,Clarkdoyle2013}, in particular for their many applications to the imaging modalities of Computerized tomography and Positron Emission Tomography: motion artifact reduction as in, e.g., \cite{Welch1998,Yu2007}; consistent extrapolation of truncated data as in, e.g., \cite{VanGompel2006, Xu2010}; monitoring of CT systems for faulty detector channels \cite{Patch2001}; see also the detailed introduction in \cite{Clackdoyle2013} on the matter. In the present article, we show that another range characterization provided in \cite{Pestov2004} is {\em equivalent}\footnote{It is, in fact, a generalization, since parallel coordinates do not exist on surfaces without symmetries.} to the moment conditions for functions with compact support. Additionally, appropriate boundary operators (some of them, labelled as $P_\pm$ first appearing in \cite{Pestov2004}; others to be introduced below, found by the author) allow to enforce in a fast way {\em all} moment conditions simultaneously. 
Such operators involve the interplay between the scattering relation and the Hilbert transform on the tangent circles at the boundary of $M$. 

It is the author's hope that such an equivalence bridges a gap between the two systems of coordinates and offers a new viewpoint on the X-ray transform and the moment conditions in fan-beam coordinates, whose theoretical focus in the Euclidean setting is often shifted to parallel coordinates. This is because parallel geometry, unlike fan-beam, enjoys the Central Slice Theorem, which allows for a deeper understanding of discretization issues and regularization theory via accurate and efficient convolution methods in data space \cite{Natterer2001,Epstein2008}. This progress in fan-beam coordinates is partly motivated by the fact that dealing with such coordinates becomes almost unavoidable when considering the TTP on Riemannian surfaces without symmetries.


Finally, the expliciteness of the present results owes much to the fact that the scattering relation on a Euclidean disk admits a very explicit expression, thereby simplifying the description of the action of the boundary operators mentioned above. A generalization of these results to simple Riemannian surfaces is currently under study and will appear in future work. 

We now discuss the main results and give an outline of the remainder of the article. 

\section{Statement of the main results}

We consider $\Dm = \{\x = (x,y)\in \Rm^2: x^2 + y^2 <1\}$ the Euclidean unit disk, with boundary $\partial \Dm = \{e^{i\beta}, \beta\in \Sm^1\}$. Here and below, we may identify at will a point $\x = (x,y)\in \Dm$ with its complex representative $z = x + iy$. Define the usual Sobolev spaces $L^2(\Dm)$, $H^1(\Dm)$ and $H^1_0(\Dm)$, as well as 
\begin{align*}
    \dot{H}^1(\Dm) := \left\{ f\in H^1(\Dm), \quad \int_{\Sm^1} f(e^{i\beta})\ d\beta = 0 \right\}.
\end{align*}
The domain $\Dm$ has a unit circle bundle $S\Dm \cong \{(\x,\theta)\in \Dm\times \Sm^1\}$ where $\theta$ identifies the unit tangent vector $\binom{\cos\theta}{\sin\theta}$ at $\x$. $S\Dm$ has influx/outflux boundaries\footnote{The reader should be aware that opposite definitions of $\partial_\pm SM$ can appear in the litterature. Here, we follow the convention as in, e.g., \cite{Pestov2004,Monard2013a,Paternain2011a}, where ``$+$'' designates ``influx''.} 
\begin{align*}
    \partial_\pm SM = \{(\x,v) \in \partial SM,\ \x\in \partial M,\ \mp v\cdot \nu_\x >0\}, 
\end{align*}
where $\nu_\x$ is the unit outer normal at $\x\in \partial \Dm$ (here we may identify $\nu_\x = \x$). We parameterize the influx boundary in {\em fan-beam coordinates}, that is, in terms of a couple $(\beta,\alpha) \in \Sm^1\times \left( -\frac{\pi}{2}, \frac{\pi}{2} \right)$, where $\beta$ describes the point at the boundary $\x(\beta) = \binom{\cos\beta}{\sin \beta}$, and $\alpha$ denotes the direction of the tangent vector with respect to $\nu_\x$, i.e. $v = \binom{\cos(\beta+\pi+\alpha)}{\sin(\beta+\pi+\alpha)}$.

We consider the {\em X-ray transform} (XRT) $I:\C^\infty(S\Dm)\to \C^\infty (\partial_+ SM)$ defined by 
\begin{align}
    If(\beta,\alpha) &= \int_{0}^{\tau(\beta,\alpha)} f(\gamma_{\beta,\alpha}(t), \arg{(\dot\gamma_{\beta,\alpha}(t))}) \ dt \nonumber \\
    &= \int_0^{2\cos\alpha} f(e^{i\beta} + te^{i(\beta+\pi+\alpha)}, \beta+\pi+\alpha) \ dt, \quad (\beta,\alpha) \in \partial_+ SM. \label{eq:XRT}
\end{align}
While this transform is usually made continous in the $L^2(SM)\to L^2(\partial_+ SM,\cos\alpha)$ setting, we prove in Lemma \ref{lem:continuity} that it is in fact $L^2(SM)\to L^2(\partial_+ SM)$ continuous. It later becomes more convenient to construct Hilbert bases in this unweighted codomain.

\noindent {\bf Transport equations on $S\Dm$.} The ray transform can be represented as the ingoing trace $If = u|_{\partial_+ SM}$ of a solution $u(\x,\theta)$ to the following transport equation
\begin{align*}
  Xu = -f \quad (S\Dm), \qquad u|_{\partial_- SM} = 0,
\end{align*}
where we have defined the geodesic vector field on $T(S\Dm)$ by
\begin{align}
    X = \cos\theta \partial_x + \sin\theta \partial_y = \eta_+ + \eta_-, \qquad \eta_+ := e^{i\theta} \partial, \qquad \eta_- = \overline{\eta_+} = e^{-i\theta} \dbar, 
  \label{eq:X}
\end{align}
and where $\partial := \frac{1}{2} (\partial_x - i\partial_y)$. Using notation from the canonical framing of the unit circle bundle $S\Dm$, we define 
\[ X_\perp := [X,\partial_\theta] = \frac{\eta_+ - \eta_-}{i} = - (-\sin\theta \partial_x + \cos\theta \partial_y). \]
The $\Sm^1$ action on $S\Dm$ induces a Fourier decomposition on $L^2(S\Dm, \ d\x\ d\theta)$: a function $u\in L^2(S\Dm)$ can be written uniquely as
\begin{align*}
    u(\x,\theta) &= \sum_{k\in \Zm} u_k(\x) e^{ik\theta}, \quad\text{where } u_k(\x):= \frac{1}{2\pi} \int_{\Sm^1} u(\x,\theta) e^{-ik\theta}\ d\theta, \\ 
    \|u\|_{L^2(S\Dm)}^2 &= 2\pi \sum_{k\in \Zm} \|u_k\|_{L^2(\Dm)}^2,
\end{align*}

In this decomposition, a diagonal linear operator of interest is the {\em fiberwise Hilbert transform} $H:L^2(S\Dm)\to L^2(S\Dm)$, whose action on each harmonic component is described as
\begin{align*}
    H (u_k(\x) e^{ik\theta}) = -i\sgn{k}\ u_k(\x) e^{ik\theta}, \quad k\in \Zm\quad  (\text{with the convention }\sgn{0} = 0). 
\end{align*}
We write the decomposition $H= H_+ + H_-$, where $H_{+/-}$ stands for the composition of $H$ with projection onto even/odd harmonics. 

We now explain how definition \ref{eq:XRT} contains the case where one integrates tensor fields.


\noindent {\bf The correspondence between tensors and functions on $S\Dm$ with finite harmonic content.} In two dimensions, any symmetric $m$-covariant tensor may be represented as
\begin{align*}
    \bff = \sum_{k=0}^m \binom{m}{k} f_k (\x)\ \sigma(dx^{\otimes k}\otimes dy^{\otimes (m-k)}),
\end{align*}
with $\sigma$ denoting the symmetrization operator. For a curve $(\x(t), \theta(t))$ in $SM$, one traditionally defines the ray transform of $\bff$ as the integral along this curve of $\bff$ paired $m$ times with the speed vector $\binom{\cos\theta(t)}{\sin\theta(t)}$. This amounts to integrating the following function on $SM$ along said curve, in one-to-one correspondence with the tensor $\bff$
\begin{align*}
    f(\x,\theta) = \sum_{k=0}^m \binom{m}{k} f_k(\x) \cos^k\theta \sin^{m-k}\theta. 
\end{align*} 
This function can then be uniquely represented in the Fourier decomposition in $\theta$ solely in terms of the following harmonics
\begin{align*}
    f = \left\{
    \begin{array}{lr}
	f_0 + f_{\pm 2}e^{\pm 2i\theta} + \cdots + f_{\pm m} e^{\pm im\theta} &\quad (m \text{ even}), \\
	f_{\pm 1}e^{\pm i\theta} + f_{\pm 3}e^{\pm 3i\theta} + \cdots + f_{\pm m}e^{\pm im\theta} &\quad (m \text{ odd}).
    \end{array}
    \right.
\end{align*}
For convenience, we still call such elements $m$-tensors and define $S^m$ the subspace of $m$-tensors in $L^2(S\Dm)$. In this identification, we have the following correspondences $\bd \bff \leftrightarrow Xf$, and for $g\in H^1(\Dm)$, $\bd g \equiv dg \leftrightarrow Xg$ and $\star dg \leftrightarrow X_\perp g$ ($\star$: Hodge star operator). In particular, the Helmholtz-Hodge decomposition of one-forms reads as follows: a function $w\in L^2(S\Dm)$ of the form $w = w_1 + w_{-1}$ decomposes uniquely in the form 
\begin{align}
    w = Xf + X_\perp g, \qquad f\in H^1_0(\Dm), \qquad g\in \dot{H}^1(\Dm). 
    \label{eq:Hodge}
\end{align}

\noindent {\bf Main results.} The decomposition problem mentioned in the introduction is decoupled between even and odd tensors. For an integer $k$, we denote:
\begin{align}
    \begin{split}
	\ker^k \eta_+ &= \{ h(\x,\theta) = f(\x) e^{ik\theta}, \quad f\in L^2(\Dm), \quad \partial f = 0 \}, \\
	\ker^k \eta_- &= \{ h(\x,\theta) = f(\x) e^{ik\theta},\quad f\in L^2(\Dm), \quad \dbar f = 0 \}.
    \end{split}    
    \label{eq:kerk}
\end{align}


\begin{theorem}[Decomposition]\label{thm:decomp}
  For any $m$-tensor $f\in L^2(SM)$, there exists a unique $m$-tensor $g\in L^2(SM)$ such that $If = Ig$, and $g$ is of the following form: 
  \begin{itemize}
      \item[(i)] if $m=2n$ is even, then $g = g_0 + g_{\pm 2}e^{\pm i2\theta} + \cdots + g_{\pm 2n}e^{\pm i2n\theta}$ with $g_0\in L^2(M)$ 
	  and $(g_{\pm k} e^{\pm ik\theta}) \in \ker^{\pm k} \eta_{\mp}$ for all $2\le k\le n$, $k$ even. 
      \item[(ii)] If $m=2n+1$ is odd, then $g = X_\perp g_0 + g_{\pm 3}e^{\pm i3\theta} + \cdots + g_{\pm (2n+1)}e^{\pm i(2n+1)\theta}$, with $g_0\in \dot{H}^1(M)$ 
	  and $(g_{\pm k}e^{i\pm k\theta}) \in \ker^{\pm k} \eta_{\mp}$  for all $3\le k\le n$, $k$ odd. 
  \end{itemize}
  In both statements above, there exists a constant $C_m$ independent of $f,g$ such that 
  \begin{align*}
    \|g\|_{L^2(SM)} \le C_m\|f\|_{L^2(SM)}.
  \end{align*}
\end{theorem}

As mentioned earlier, this decomposition appears naturally in the context of finding conformal Killing tensor fields on manifolds, see e.g. \cite{Sharafudtinov2007,Sharafutdinov2014} or \cite[Theorem 1.5]{Dairbekov2011}. An advantage of this representation is that the ray transforms of each component of $g$ live on orthogonal subspaces of $L^2(\partial_+ SM)$, as explained in the following theorem. For $f\in L^2(M)$, we define $I_0 f := I[f]$ (i.e. regard $f$ as a function on $SM$ constant in $\theta$), and for $h\in \dot{H}^1(M)$, we define $I_\perp h = I[X_\perp h]$.

\begin{theorem}[Range decomposition]\label{thm:range}
    For any natural integer $m$, we have the following range decompositions, orthogonal in $L^2(\partial_+ SM)$:
    \begin{align*}
	I(S^{2m}) &= I_0(L^2(M)) \stackrel{\perp}{\oplus} \bigoplus_{k=1}^m I(\ker^{2k} \eta_-) \stackrel{\perp}{\oplus} \bigoplus_{k=1}^{m} I(\ker^{-2k} \eta_+) \\
	I(S^{2m+1}) &= I_\perp( \dot{H}^1(M) ) \stackrel{\perp}{\oplus} \bigoplus_{k=1}^m I(\ker^{2k+1} \eta_-) \stackrel{\perp}{\oplus} \bigoplus_{k=1}^{m} I(\ker^{-(2k+1)} \eta_+).
    \end{align*}
\end{theorem}

Another characterization of the range of the X-ray transform over tensors was provided in \cite{Paternain2013a} on simple surfaces, where the sum there need not be direct as in the decomposition above. In the range decomposition above, the largest subspaces are the ranges of $I_0$ and $I_\perp$, whose description benefits from the Pestov-Uhlmann range characterizations appearing in \cite{Pestov2004}. More precisely, upon defining operators $P_\pm:\C^\infty(\partial_+ SM)\to \C^\infty(\partial_+ SM)$ in terms of the Hilbert transform and the scattering relation, it is proved in \cite{Pestov2004} that the ranges of $I_0$ and $I_\perp$ match those of $P_-$ and $P_+$ in smooth topologies, respectively (see Section \ref{sec:PUHL} for detail). An earlier characterization of the range of $I_0$ in parallel coordinates, in the form of moment conditions, was found by Gelfand, Graev \cite{GelfandGraev1960}, and separately by Helgason and Ludwig \cite{Helgason1999,Ludwig1966}. We prove here that, translated into fan-beam coordinates for compactly supported functions, they become equivalent to the Pestov-Uhlmann range characterization. 

\begin{theorem} \label{thm:equivalence}
    For compactly supported functions, the moment conditions for the two-dimensional Radon transform are equivalent to the Pestov-Uhlmann range characterization of $I_0$ in the case of a Euclidean disk. 
\end{theorem}

Both theorems \ref{thm:range} and \ref{thm:equivalence} make use of explicit bases of $L^2 (\partial_+ SM)$ introduced in Section \ref{sec:bases}. The orthogonality of the ranges in Theorem \ref{thm:range} makes it especially simple to separate the components in data space, and to reconstruct them separately and efficiently. More specifically, the corresponding reconstruction procedure is given below (see Sec. \ref{sec:scatrel} for definition of $A_\pm$ and $A_\pm^\star$, and Eq. \ref{eq:backproj} for the definition of $I_0^\sharp$ and $I_\perp^\sharp$). In the statement below, we denote by $\dot{H}^1(\ker^0 \eta_\pm,M) = \dot{H}^1(M)\cap \ker^0 \eta_{\pm}$. 

\begin{theorem}[Reconstruction]\label{thm:recons} Let $f\in L^2(SM)$ an $m$-tensor and $g$ as in Theorem \ref{thm:decomp}, and denote the data $\D = If = Ig$.
    \begin{itemize}
	\item[(i)] If $m=2n$ is even, the functions $g_0, g_{\pm 2}, \cdots, g_{\pm 2n}$ can be uniquely reconstructed by means of the following formulas:
	    \begin{align}
		g_0  &= \frac{1}{8\pi} I_\perp^\sharp A_+^\star H A_- \D, \qquad \text{and for}\quad  1\le k\le n, \label{eq:reconsI0}\\
		g_{2k}(z) &= \frac{1}{2\pi^2} \int_{\Sm^1} \frac{e^{-i2k\beta}}{(1-ze^{-i\beta})^2} \int_{-\frac{\pi}{2}}^{\frac{\pi}{2}} \D (\beta,\alpha) e^{i(1-2k)\alpha} \ d\alpha\ d\beta,  \label{eq:reconsholo1} \\
		g_{-2k}(z) &= \frac{1}{2\pi^2} \int_{\Sm^1} \frac{e^{i2k\beta}}{(1-\bar{z} e^{i\beta})^2} \int_{-\frac{\pi}{2}}^{\frac{\pi}{2}} \D (\beta,\alpha) e^{i(-1+2k)\alpha} \ d\alpha\ d\beta. \label{eq:reconsholo2}
	    \end{align}
	\item[(ii)] If $m=2n+1$ is odd, the functions $g_0$, $g_{\pm 3}$, \ldots, $g_{\pm(2n+1)}$ can be uniquely reconstructed by means of the formulas: decomposing the function $g_0$ into $g_0 = g_{(0)} + g_{(+)} + g_{(-)}$ where $g_{(0)}\in H^1_0(M)$, $g_{(\pm)} \in \dot{H}^1(\ker^0 \eta_\pm,M)$, we have
	    \begin{align}
	      g_{(0)}  &= \frac{-1}{8\pi} I_0^\sharp A_+^\star H A_- \left(Id + (A_-^\star H A_-)^2\right)\D,  \label{eq:reconsgzero}\\
	      g_{(-)}(z) &= \frac{1}{2i\pi^2} \int_{\partial_+ SM} \D(\beta,\alpha) \frac{ze^{-i\beta}}{1-ze^{-i\beta}}\ d\alpha\ d\beta, \label{eq:reconsgminus} \\
	      g_{(+)}(z) &= \frac{i}{2\pi^2} \int_{\partial_+ SM} \D(\beta,\alpha) \frac{\zbar e^{i\beta}}{1-\zbar e^{i\beta}}\ d\alpha\ d\beta, \qquad \text{and for}\quad  1\le k\le n, \label{eq:reconsgplus} \\
	      g_{2k+1}(z) &= \frac{1}{2\pi^2} \int_{\Sm^1} \frac{e^{-i(2k+1)\beta}}{(1-ze^{-i\beta})^2} \int_{-\frac{\pi}{2}}^{\frac{\pi}{2}} \D (\beta,\alpha) e^{-i2k\alpha} \ d\alpha\ d\beta, \label{eq:reconsholo3} \\
	      g_{-2k-1}(z) &= \frac{1}{2\pi^2} \int_{\Sm^1} \frac{e^{i(2k+1)\beta}}{(1-\bar{z} e^{i\beta})^2} \int_{-\frac{\pi}{2}}^{\frac{\pi}{2}} \D (\beta,\alpha) e^{i2k\alpha} \ d\alpha\ d\beta. \label{eq:reconsholo4}
	    \end{align}
    \end{itemize}    
\end{theorem}

\begin{remark}
  \begin{itemize}
      \item[$(i)$] The main ideas in deriving \ref{eq:reconsI0} and \ref{eq:reconsgzero} first appeared in \cite{Pestov2004}, later modified using the operator $A_+^\star H A_-$ in \cite[Proposition 2.2]{Monard2015}. Such formulas only inverted $I_\perp$ for functions vanishing at the boundary. Inversion of $I_\perp$ over integrands which may be supported up the the boundary in the present paper leads us to introduce the factor $Id + (A_-^\star H A_-)^2$ in \ref{eq:reconsgzero}. 
      \item[$(ii)$] Each formula above contains a projection onto the appropriate subspace before reconstruction of the corresponding independent component, so that it is not necessary to pre-process the data $\D$ before applying these formulas. 
      \item[$(iii)$] The main bulk of the reconstruction, containing all possible visible singularities, is contained in the $g_0$ mode, involving the inversion of $I_0$ or $I_\perp$. The additional analytic terms account for amplitude discrepancies in the resulting data, though they do not contain any singularities. These additional terms, unless identically zero, are never compactly supported inside $M$, even when the initial tensor $f$ is, which was a feature already present in the case of the solenoidal representative discussed in the introduction.
      \item[$(iv)$] The decompositions presented above rely on the choice of a particular ``central harmonic'' $g_0$ for the reconstructible representative. We discuss in Section \ref{sec:decomps} the possibility of other decompositions and representatives based on changing the central harmonic, and we briefly discuss how to reconstruct them.
  \end{itemize}  
\end{remark}

\noindent{\bf Outline.} Section \ref{sec:decomp} covers the proof of Theorem \ref{thm:decomp}. Section \ref{sec:range} studies the range decomposition of the X-ray transform: we first introduce appropriate bases for $L^2(\partial_+ SM)$ in Sec. \ref{sec:bases}, then study the scattering relation and other operators based on it in Sec. \ref{sec:Vplusminus} and \ref{sec:scatrel}; Sec. \ref{sec:PUHL} covers the proof of Theorem \ref{thm:equivalence}; Sec. \ref{sec:Iperp} covers the range decomposition of $I_\perp$, and Sec. \ref{sec:otherranges} covers the range of $I$ over any space $\ker^m \eta_\pm$, altogether proving Theorem \ref{thm:range}. Section \ref{sec:recons} covers all the reconstruction formulas, as described in Theorem \ref{thm:recons}, first treating the case of $I_0$ and $I_\perp$ in Sec. \ref{sec:I0Iperp} and then considering the additional analytic terms in Sec. \ref{sec:other}. Section \ref{sec:decomps} discusses other possible decompositions and reconstruction strategies, and Section \ref{sec:numerics} covers numerical examples. 

\section{Decomposition of a tensor - Proof of Theorem \ref{thm:decomp}.} \label{sec:decomp}

The proof of Theorem \ref{thm:decomp} relies on viewing the X-ray transform as the influx trace of a solution to certain transport equation posed on $S\Dm$ as in the introduction, and changing both solution and source term without altering the boundary values. We start by stating the following lemma, whose proof is standard and omitted.
\begin{lemma} \label{lem:decomp}
  Let $f\in L^2(M)$.
  \begin{itemize}
      \item[(i)] There exists a unique couple $(v, g) \in H^1_0(M)\times L^2(M)$ such that $f = \partial v + g$ and $\dbar g = 0$, satisfying the stability estimate 
	  \begin{align*}
	      \|v\|_{H^1_0(M)} + \|g\|_{L^2(M)} \le C \|f\|_{L^2(M)},
	  \end{align*} 
	  for some constant $C$ independent of $f$.
      \item[(ii)] By complex conjugation, there also exists a unique couple $(v', g') \in H^1_0(M)\times L^2(M)$ such that $f = \dbar v' + g'$ with $\partial g' = 0$, satisfying the same estimate as above.
  \end{itemize}
\end{lemma}

\begin{remark}
    In the context of Riemannian surfaces, the actual decomposition to look at is, for every $\pm k >0$ the splitting of a smooth element $f\in\ker (\partial_\theta - ikId)$ into $f = \eta_\pm v + g$ with $v|_{\partial SM} = 0$ and $\eta_\mp g = 0$, see \cite{Guillemin1980}. A reader familiar with these decompositions may notice that the proof of Theorem \ref{thm:decomp} holds for geodesic ray transforms on any Riemannian surface where such decompositions exist.    
\end{remark}

The proof of Theorem \ref{thm:decomp} is then based on an iterative use of these decompositions.

\begin{proof}[Proof of Theorem \ref{thm:decomp}.] {\bf Proof of $(i)$.} We treat a tensor with only non-negative harmonic content and explain how to generalize, i.e., we prove the case of an even tensor of the form $f = f_0 + f_{2}e^{i2\theta} + \cdots + f_{2n}e^{i2n\theta} \in L^2(SM)$. Recall the transport equation 
    \begin{align*}
	Xu = (e^{i\theta} \partial + e^{-i\theta} \dbar) u = -f, \qquad u|_{\partial_- SM} = 0, \qquad u|_{\partial_+ SM} = If. 
    \end{align*}
    Using lemma \ref{lem:decomp}(i), since $f_{2n}\in L^2(M)$, we decompose $f_{2n} = \partial v_{2n-1} + g_{2n}$ with $v_{2n-1}|_{\partial M} = 0$ and $\dbar g_{2n} = 0$. We then rewrite this as an equality of functions on $SM$: 
    \begin{align*}
	f_{2n} e^{i2n\theta} &= (e^{i\theta} \partial) (e^{i(2n-1)\theta}v_{2n-1}) + e^{i2n\theta} g_{2n} \\
	&= X (e^{i(2n-1)\theta}v_{2n-1}) - e^{i(2n-2)\theta}\dbar v_{2n-1} + e^{i2n\theta} g_{2n},
    \end{align*}
    so that the transport equation may be rewritten as
    \begin{align*}
	X (u+e^{i(2n-1)\theta}v_{2n-1}) = - (f_0 + f_{2} \cdots + e^{i(2n-2)\theta} (f_{2n-2}-\dbar v_{2n-1})  + e^{i2n\theta} g_{2n}),
    \end{align*}
    where $\dbar g_{2n} = 0$, $f_{2n-2} - \dbar v_{2n-1} \in L^2(M)$ and all components are controlled by $\|f\|_{L^2(SM)}$. Since $v_{2n-1}|_{\partial M} = 0$, integrating this equation along all straight lines gives the data $If$, where the right-hand-side now has a top term satisfying $\dbar g_{2n} = 0$. One may repeat this process for the term of harmonic order $2n-2$ and $n-2$ more times, to arrive at a tensor $g$ satisfying $Ig = If$, of the form
    \begin{align*}
	g = g_0(\x) +  g_2(\x)e^{i2\theta} + \cdots +  g_{2n}(\x) e^{i2n\theta},
    \end{align*}
    where $\dbar g_{2k} = 0$ for $0<k\le n$ with an estimate $\|g\|_{L^2(M)} \le C \|f\|_{L^2(M)}$, for a constant $C$ independent of $f$. 

    Now if $f$ has negative harmonic terms, we may use lemma \ref{lem:decomp}(ii) to decompose the negative harmonics in a similar fashion and arrive at the desired result. 

    {\bf Proof of $(ii)$.} Similarly to the first case, if $f$ is of the form $f = f_{\pm1}e^{\pm i\theta} + \cdots + f_{\pm (2n+1)}e^{\pm i(2n+1) \theta}$, we may use Lemma \ref{lem:decomp} iteratively to construct a unique 
    \begin{align*}
	f' = g'_{\pm 1} e^{\pm i\theta} +  g'_{\pm 3} e^{\pm 3i\theta} + \cdots g'_{\pm (2n+1)} e^{\pm i(2n+1)\theta}, 
    \end{align*}
    such that $\|f'\|_{L^2(SM)} \le C \|f\|_{L^2(SM)}$ with $C$ independent of $f$ and $\partial g_{-(2k+1)} = \dbar g_{2k+1} = 0$ for $1\le k\le n$. We then apply the Helmholtz-Hodge decomposition to the one-form $g'_{- 1} e^{- i\theta} + g'_{1} e^{i\theta}$ to write it uniquely in the form 
    \begin{align*}
	g'_{- 1} e^{- i\theta} + g'_{1} e^{i\theta} = Xv + X_\perp g_0, \qquad v\in H^1_0(M), \qquad g_0\in \dot{H}^1(M),
    \end{align*}
    so that, the transport equation $Xu = -f'$ may be recast as 
    \begin{align}
	X(u+v) = - (X_\perp g_0 + g'_{\pm 3} e^{\pm 3i\theta} + \cdots g'_{\pm (2n+1)} e^{\pm i(2n+1)\theta}),
	\label{eq:ttt}
    \end{align}
    where $(u+v)|_{\partial SM} = u|_{\partial SM}$. The right-hand side is the desired decomposition. The vanishing of $v$ at $\partial M$ implies that the right-hand side of \ref{eq:ttt} has the same ray transform as $f$. Theorem \ref{thm:decomp} is thus proved. 
\end{proof}

\section{Range decomposition - Proofs of Theorems \ref{thm:range} and \ref{thm:equivalence}} \label{sec:range}

For $(\beta,\alpha)\in \partial_+ SM$, denoting $\varphi_{\beta,\alpha}(t) = (e^{i\beta} + t e^{i(\beta+\pi+\alpha)}, \beta+\pi+\alpha)$ the geodesic flow, Santal\'o's formula reads in our context:
\begin{align*}
  \int_{SM} f(\x,\theta) \ d\x\ d\theta = \int_{\partial_+ SM} \cos\alpha \int_0^{2\cos\alpha} f(\varphi_t(\beta,\alpha))\ dt\ d\beta\ d\alpha.
\end{align*}
This formula usually suggests that the ray transform is naturally defined in the setting $I:L^2(SM)\to L^2(\partial_+ SM, \cos\alpha)$, see e.g. \cite{Sharafudtinov1994}. However, when considering the fiberwise Hilbert transform later, the weight $\cos\alpha$ may be bothersome in the sense that on $L^2(\Sm^1, |\cos\alpha|)$, the Hilbert transform cannot be made continuous \cite{Petermichl2002}. We will in fact establish that this weight can be removed so that we can work in a setting where the operator $H$ is well-behaved.

\begin{lemma}
  The ray transform $I:L^2(SM) \to L^2(\partial_+ SM)$ is bounded and $\|I\| = 2$.
  \label{lem:continuity}
\end{lemma}

\begin{proof} Using Cauchy-Schwarz inequality, we write that 
  \begin{align*}
    |If(\beta,\alpha)|^2 = \left| \int_{0}^{2\cos\alpha} f(\phi_t(\beta,\alpha))\ dt \right|^2 \le 2\cos\alpha \int_{0}^{2\cos\alpha} |f(\phi_t(\beta,\alpha))|^2\ dt.
  \end{align*}
  Integrating the inequality over $\partial_+ SM$ and using Santal\'o's formula, we obtain that 
  \begin{align*}
    \int_{\partial_+ SM} |If(\beta,\alpha)|^2 \ d\beta\ d\alpha \le 2 \|f\|^2_{L^2(SM)},
  \end{align*}
  where inequality becomes equality in the case $f\equiv 1$. Hence the result.
\end{proof}

\subsection{Two Hilbert bases for $L^2(\partial_+ SM)$.}\label{sec:bases}
In fan-beam coordinates, $L^2(\partial_+ SM)$ can be regarded as $L^2\left(\Sm^1\times \left( -\frac{\pi}{2}, \frac{\pi}{2} \right)\right)$, for which we define a Hilbert orthonormal basis
\begin{align*}
    \B = \left\{\phi_{k,l}(\beta,\alpha) = \frac{1}{\pi\sqrt{2}} e^{i(k\beta+2l\alpha)}, \qquad k,l\in \Zm^2\right\}.
\end{align*}
For any couple $(k,l)\in \Zm^2$, we also define $\phi_{k,l}'(\beta,\alpha) = e^{i\alpha} \phi_{k,l}(\beta,\alpha)$, and define a second orthonormal basis of $L^2(\partial_+ SM)$
\begin{align*}
    \B' = \left\{ \phi'_{k,l}, \quad (k,l)\in \Zm^2   \right\}.
\end{align*}
The main reason for introducing two bases is that we will later need to extend functions on $\partial_+ SM$ into functions on $\partial SM$ by evenness or oddness w.r.t. the involution $\alpha\mapsto \alpha+\pi$, and depending on the representation, such processes simply consist of extending the expression at hand to $\alpha\in \Sm^1$. The change of basis $\B\leftrightarrow\B'$ is {\em non-local}: indeed, the function $\alpha\mapsto e^{i\alpha}$ decomposes into the basis $\B$ as follows
\begin{align*}
  e^{i\alpha} = 2\sqrt{2} \sum_{l\in \Zm} \frac{(-1)^l}{1-2l} \phi_{0,l}, \qquad e^{-i\alpha} = 2\sqrt{2} \sum_{l\in \Zm} \frac{(-1)^l}{1-2l} \phi_{0,-l},
\end{align*}
so that, using the property that $\phi_{p,q}\phi_{p',q'} = \phi_{p+p', q+q'}$, the formulas for changing basis are given by 
\begin{align*}
    \phi'_{p,q} = 2\sqrt{2} \sum_{l\in \Zm} \frac{(-1)^l}{1-2l} \phi_{p,l+q}, \quad\text{and}\quad \phi_{k,l} = e^{-i\alpha} \phi'_{k,l} = 2\sqrt{2} \sum_{q\in \Zm} \frac{(-1)^q}{1-2q} \phi'_{k,l-q}.
\end{align*}

\subsection{Scattering relations and the $\V_+/\V_-$ decomposition} \label{sec:Vplusminus}

We define the {\em scattering relation}\footnote{In the litterature, it may also be defined by $\alpha$, a notation that is avoided here not to conflict with fan-beam coordinate notation.} $\SS:\partial SM\to \partial SM$ and the {\em antipodal scattering relation} $\SS_A:\partial_+ SM\to \partial_+ SM$ by 
\begin{align}
    \SS(\beta,\alpha) := (\beta+\pi+2\alpha, \pi-\alpha), \qquad \SS_A (\beta,\alpha) = (\beta+\pi+2\alpha,-\alpha). \label{eq:ASR}
\end{align}
Note that $\SS:\partial_\pm SM\to \partial_\mp SM$ and that $\SS_A$ is the composition of $\SS$ with the antipodal map $\alpha\mapsto \alpha+\pi$. Both relations are involutions of their respective domains and we consider their associated pull-backs $\SS_A^\star f:= f\circ \SS_A$ and $\SS^\star g := g\circ \SS$ for $f\in L^2(\partial_+ SM)$ and $g\in L^2(\partial SM)$. The operators $\SS_A^\star:L^2(\partial_+ SM)\to L^2(\partial_+ SM)$ and $\SS^\star:L^2(\partial SM)\to L^2(\partial SM)$ are self-adjoint, so when using these operators, we reserve the $^\star$ notation for pull-backs. For basis elements of $\B,\B'$, it is straightforward to establish the identities,
\begin{align}
    \SS_A^\star \phi_{p,q} = (-1)^p \phi_{p,p-q}, \qquad \SS_A^\star \phi'_{p,q} = (-1)^p \phi'_{p,p-q-1}, \qquad (p,q)\in \Zm^2.	
    \label{eq:SAstar}    
\end{align}
When the expressions $\phi_{p,q}$ and $\phi'_{p,q}$ are regarded as elements on $L^2(\partial SM)$, we also have the identities
\begin{align}
    S^\star \phi_{p,q} = (-1)^p \phi_{p,p-q}, \qquad S^\star \phi'_{p,q} = - (-1)^p \phi'_{p,p-q-1}, \qquad (p,q)\in \Zm^2.	
    \label{eq:Sstar}
\end{align}

Identity \ref{eq:SAstar} makes $S_A^\star$ an isometry of $L^2(\partial_+ SM)$. Moreover, we have the following orthogonal splitting
\begin{align}
  L^2(\partial_+ SM) = \V_+ \stackrel{\bot}{\oplus} \V_-, \qquad \V_\pm := \ker (Id \mp S_A^\star).
  \label{eq:splitting}
\end{align}
Out of the bases $\B$ and $\B'$, we can then extract two distinct bases for each space $\V_+$ and $\V_-$ by successively applying the operators $Id \pm \SS_A^\star$ and removing redundancies in $(p,q)$. Given the formulas \ref{eq:SAstar}, let us define
\begin{align*}
    u_{p,q} &:= (Id + \SS_A^\star) \phi_{p,q} = \phi_{p,q} + (-1)^p \phi_{p,p-q}, \qquad p \le 2q, \\
    v_{p,q} &:= (Id - \SS_A^\star) \phi_{p,q} = \phi_{p,q} - (-1)^p \phi_{p,p-q}, \qquad p < 2q, \\
    u'_{p,q} &:= (Id + \SS_A^\star) \phi'_{p,q} = \phi'_{p,q} + (-1)^p \phi'_{p,p-q-1}, \qquad p < 2q+1, \\
    v'_{pq} &:= (Id - \SS_A^\star) \phi'_{p,q} = \phi'_{p,q} - (-1)^p \phi'_{p,p-q-1}, \qquad p \le 2q+1.
\end{align*}
We obtain two bases for each subspace:
\begin{align}
    \begin{split}
	\V_+ &= \text{span } (u_{p,q}, \quad p\le 2q) = \text{span } (u'_{p,q}, \quad p<2q+1), \\
	\V_- &= \text{span } (v_{p,q}, \quad p < 2q) = \text{span } (v'_{p,q}, \quad p \le 2q+1).	
    \end{split}
    \label{eq:bases}    
\end{align}

All definitions considered are valid for all $(p,q)$, though with the following redundancies: 
\begin{align*}
    u_{p,p-q} = (-1)^p u_{p,q}, &\qquad v_{p,p-q} = -(-1)^p v_{p,q} \\
    u'_{p,p-q-1} = (-1)^p u'_{p,q} &\qquad v'_{p,p-q-1} = -(-1)^p v'_{p,q}.
\end{align*}
We now summarize how these bases transform under complex conjugation: using the property $\overline{\phi_{p,q}} = \phi_{-p,-q}$ true for every $(p,q)$, we have
\begin{align}
    \begin{split}
	\overline{u_{p,q}} = u_{-p,-q} = (-1)^p u_{-p,-p+q}, &\qquad \overline{v_{p,q}} = v_{-p,-q} = -(-1)^p v_{-p,-p+q}, \\
	\overline{u'_{p,q}} = u_{-p,-q-1} = (-1)^p u'_{-p,-p+q}, &\qquad \overline{v'_{p,q}} = v_{-p,-q-1} = -(-1)^p v'_{-p,-p+q},
    \end{split}
    \label{eq:conj}
\end{align}
where the pairs of indices in the last column are in the reduced $(p,q)$ ranges as in \ref{eq:bases}. The next identities summarize how the basis elements transform under $Id - \SS^\star$: 
\begin{align}
    \begin{split}
	(Id-\SS^\star) \phi_{p,q} = v_{p,q}, &\qquad (Id-\SS^\star) (-1)^p \phi_{p,p-q} = -v_{p,q} \\
	(Id-\SS^\star) \phi'_{p,q} = u'_{p,q}, &\qquad (Id-\SS^\star) (-1)^p \phi_{p,p-q-1} = u'_{p,q}.
    \end{split}
    \label{eq:tempid}
\end{align}

\subsection{The operators $A_\pm$, $A_\pm^\star$, $P_\pm$ and $C_\pm$} \label{sec:scatrel}

Define the following operators $A_\pm$ of even/odd extension with respect to $\SS$, i.e. for $f$ smooth on $\partial_+ SM$, 
\begin{align*}
    A_\pm f(\beta,\alpha) = \left\lbrace 
    \begin{array}{lr}
	f(\beta, \alpha), & \text{if } (\beta,\alpha)\in \partial_+ SM, \\
	\pm f(\SS(\beta,\alpha)), & \text{if } (\beta,\alpha)\in \partial_- SM.
    \end{array}
    \right.    
\end{align*}
Such operators extend continuously to the $L^2 (\partial_+ SM, \cos\alpha) \to L^2 (\partial SM, |\cos\alpha|)$ setting (see e.g. \cite{Pestov2004}), though here, since the boundary is strictly convex, $A_\pm$ are also continuous in the $L^2(\partial_+ SM)\to L^2(\partial SM)$ setting. Moreover, because $|\text{Jac } \SS| \equiv 1$ here, the adjoints in {\em either} functional setting are given by 
\begin{align*}
    A_-^\star f(\beta,\alpha) = f(\beta,\alpha) \pm f(\SS(\beta,\alpha)), \qquad (\beta,\alpha)\in \partial_+ SM.
\end{align*}

In terms of these operators, we now define the operator $P:= A_-^\star H A_+$, which upon splitting the Hilbert transform into $H = H_+ + H_-$, yields the decomposition $P = P_+ + P_-$ with $P_\pm = A_-^\star H_\pm A_+$. A straightforward study of symmetries shows that 
\begin{align*}
    P_+ (\V_-) = 0, \qquad P_+(\V_+)\subset \V_-, \qquad P_- (\V_+) = 0, \qquad P_-(\V_-)\subset \V_+.
\end{align*}
In matrix notation in the $\V_+ \oplus \V_-$ decomposition, one may think of $P$ as $P = \left[\begin{smallmatrix} 0 & P_- \\ P_+ & 0 \end{smallmatrix}  \right]$.
In terms of $P_\pm$ operators, the following range characterizations were proved in \cite{Pestov2004} (see also \cite{Pestov2005}): for $w\in \C^\infty(\partial_+ SM)$, we have
\begin{itemize}
    \item $w\in \text{Range } I_0$ if and only if $w = P_- h$ for some $h\in \C^\infty(\partial_+ SM)$ with $A_+ h\in \C^\infty(\partial SM)$. 
    \item $w\in \text{Range } I_\perp$ if and only if $w = P_+ h$ for some $h\in \C^\infty(\partial_+ SM)$ with $A_+ h\in \C^\infty(\partial SM)$. 
\end{itemize}
Such results hold for simple surfaces with boundary, although in that context, it remains open at present how to use this characterization effectively. In the present Euclidean context, we now explain how this characterization can be used. 

The bases introduced in Section \ref{sec:bases} make it easy to compute the singular value decompositions of $P_\pm$, which in turn tells us about the harmonic content of the ranges of $I_0$ and $I_\perp$. (In all cases of nonzero spectral values below, all vectors have norm either $1$ or $\sqrt{2}$ though we normalize them with the $\ \widehat{}\ $ notation to avoid ambiguity in the spectral values.) 

\begin{proposition}\label{prop:SVDP}
    The singular value decomposition for $P_+: \V_+\to \V_-$ is given by: for every $(p,q)$ satisfying $p\le 2q$: 
    \begin{align*}
	P_+ \widehat{u_{p,q}} = \left\lbrace
	\begin{array}{lr}
	    -2i\ \widehat{v_{p,q}} & \text{if } q>0 \text{ and } p<q, \\
	    -i\ \widehat{v_{p,q}} & \text{if } q>0 \text{ and } p=q, \\
	    i\ \widehat{v_{p,q}} & \text{if } q=0 \text{ and } p<0, \\
	    0 & \text{otherwise.}
	\end{array}
	\right.
    \end{align*}
    The singular value decomposition for $P_-: \V_-\to \V_+$ is given by: for every $(p,q)$ satisfying $p<2q+1$: 
    \begin{align*}
	P_- \widehat{v'_{p,q}} = \left\lbrace
	\begin{array}{lr}
	    -2i\ \widehat{u'_{p,q}} & \text{if } q>\frac{-1}{2} \text{ and } p<q+\frac{1}{2}, \\
	    0 & \text{otherwise.}
	\end{array}
	\right.
    \end{align*}    
\end{proposition}

\begin{proof} Let $E_\pm:L^2(\partial_+ SM)\to L^2(\partial SM)$ the operators of even/odd extension with respect to the antipodal map $\alpha\mapsto \alpha+\pi$. It is important to note for the sequel that 
    \begin{itemize}
	\item[$(i)$] $E_\pm\equiv A_\pm$ on $\V_+$.
	\item[$(ii)$] $E_\pm \equiv A_{\mp}$ on $\V_-$.
	\item[$(iii)$] Applying $E_+$ to an expression in the basis $\B$ consists of extending this expression to $\alpha\in \Sm^1$.
	\item[$(iv)$] Applying $E_-$ to an expression in the basis $\B'$ consists of extending this expression to $\alpha\in \Sm^1$.
    \end{itemize}  
    Using the observations $(i)$-$(iv)$ above, we are then able to derive the following alternative definitions for $P_\pm$: 
    \begin{align*}
	P_+ f &= ((Id-\SS^\star) H E_+ f)|_{\partial_+ SM}, \quad f\in \V_+, \\
	P_- g &= ((Id-\SS^\star) H E_- g)|_{\partial_+ SM}, \quad g\in \V_-.
    \end{align*}
    These expressions, together with remarks $(iii)-(iv)$ above, suggest that the action of $P_+$ is best described in the basis $\B$ while that of $P_-$ is best described in the basis $\B'$. Using identities \ref{eq:tempid}, we then compute
    \begin{align*}
	P_+ u_{p,q} &= (Id-\SS^\star)H E_+ u_{p,q} \\
	&= -i (Id-\SS^\star) (\sgn{2q} \phi_{p,q} + \sgn{2(p-q)} (-1)^p \phi_{p,p-q}) \\
	&= -i(\sgn{2q} - \sgn{2(p-q)}) v_{p,q}.
    \end{align*}
    Similarly, we have for $P_-$
    \begin{align*}
	P_- v'_{p,q} &= (Id-\SS^\star)H E_- v'_{p,q} \\
	&= -i (Id-\SS^\star) (\sgn{2q+1}\phi'_{p,q} - \sgn{2p-2q-1} (-1)^p \phi'_{p,p-q-1}) \\
	&= -i(\sgn{2q+1}-\sgn{2p-2q-1}) u'_{p,q}.
    \end{align*}
    Studying the values of these signs by splitting cases yields the result. 
\end{proof}

Figure \ref{fig:ranges} locates the ranges of $P_\pm$ in the $(p,q)$ planes describing the spaces $\V_\pm$.

\begin{figure}[htpb]
    \centering
    \includegraphics[width=0.49\textwidth]{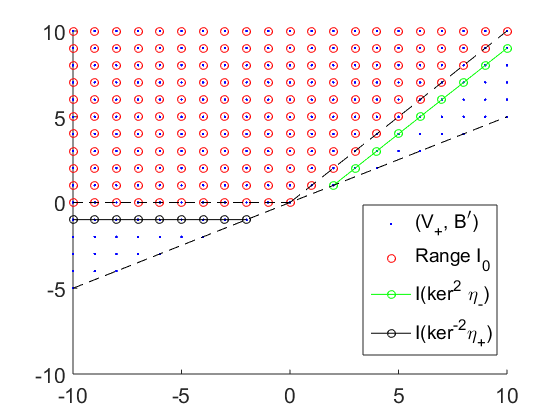}
    \includegraphics[width=0.49\textwidth]{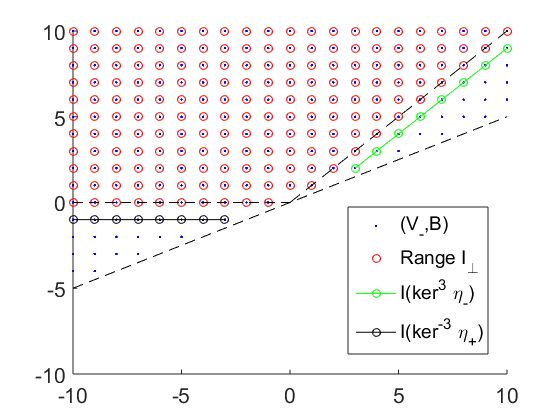}
    \caption{A visualization of the range decomposition of $I$ over even and odd tensors, sitting in $\V_+$ and $\V_-$ respectively. Both axes describe a $(p,q)$ range according to the equations defining $u'_{p,q}$ and $v_{p,q}$ as in \ref{eq:bases}. The range of $P_-$ matches that of $I_0$ (left) while the range of $P_+$ matches that of $I_\perp$ (right). On the right, the subspace $I_\perp (\dot{H}^1(\ker^0 \eta_-,M))$ is the set $\{p=q\}$ (basis elements in Range $I_\perp$ closest to $I(\ker^3 \eta_-)$), and the subspace $I_\perp (\dot{H}^1(\ker^0 \eta_+,M))$ is the set $\{q=0\}$ (basis elements in Range $I_\perp$ closest to $I(\ker^{-3} \eta_+)$).}
    \label{fig:ranges}
\end{figure}

Of practical interest is the ability to project noisy data onto the range of $I_0$. While the operator $P_-$ characterizes the range of $I_0$, it is not a projection operator and in fact annihilates the range of $I_0$ (indeed, $P_-^2 =0$ since $P_-(\V_+) = 0$ and $P_-(\V_-)\subset \V_+$). For projection purposes, let us now introduce the operator $C := \frac{1}{2} A_-^\star H A_-$. The decomposition $H = H_+ + H_-$ yields a decomposition $C = C_+ + C_-$ where $C_\pm = \frac{1}{2} A_-^\star H_\pm A_-$. This time, a direct inspection of symmetries shows that 
\begin{align*}
    C_+ (\V_+) = 0, \qquad C_+ (\V_-) \subset \V_-, \qquad C_- (\V_+) \subset \V_+, \qquad C_-(\V_-) = 0.
\end{align*}
In matrix notation in the $\V_+ \oplus \V_-$ decomposition, one may think of $C$ as $C = \left[\begin{smallmatrix} C_- & 0 \\ 0 & C_+ \end{smallmatrix} \right]$. By similar considerations as in the proof of Proposition \ref{prop:SVDP}, we can then prove that for any couple $(p,q)$, 
\begin{align*}
    C_+ v_{p,q} &= -i \frac{\sgn{2q} + \sgn{2(p-q)}}{2}v_{p,q}, \\
    C_- u'_{p,q} &= -i \frac{\sgn{2q+1}+\sgn{2p-2q-1}}{2} u'_{p,q}.  
\end{align*}

Splitting cases according to $(p,q)$, we arrive at the following
\begin{proposition}\label{prop:SVDC}
    The operators $C_-$ and $C_+$ admit the following eigenvalue decompositions:
    \begin{align*}
	&\text{For any } (p,q) \text{ with } p<2q+1, \quad C_+ v_{p,q} = \left\lbrace
	\begin{array}{lr}
	    i\ v_{p,q} & \text{if } q<0 \text{ and } p<q, \\
	    -i\ v_{p,q} & \text{if } q>0  \text{ and } p>q, \\
	    \frac{i}{2}\ v_{p,q} & \text{if } q=0 \text{ and } p<0, \\
	    \frac{-i}{2}\ v_{p,q} & \text{if } q>0 \text{ and } p=q, \\
	    0 & \text{otherwise,}
	\end{array}
	\right. \\
          &\text{For any } (p,q) \text{ with } p<2q, \quad C_- u'_{p,q} =  \left\lbrace
	\begin{array}{lr}
	    i\ u'_{p,q} & \text{if } q<\frac{-1}{2} \text{ and } p<q+\frac{1}{2}, \\
	    -i\ u'_{p,q} & \text{if } q>\frac{-1}{2} \text{ and } p>q+\frac{1}{2}, \\
	    0 & \text{otherwise.}
	\end{array}
	\right.
    \end{align*}
\end{proposition}

\subsection{Equivalence of range characterizations of $I_0$ - Proof of Theorem \ref{thm:equivalence}} \label{sec:PUHL}

The classical moment conditions state, in parallel coordinates (see, e.g., \cite{Epstein2008,Natterer2001}), that some data $\R$ is the Radon transform of some function if and only if $\R(s,\theta) = \R(-s, \theta+ \pi)$ and for every integer $n\ge 0$, the moment function $p_n(\theta) = \int_{\Rm} s^n \R(s,\theta)\ d\theta$ is a homogeneous polynomial of degree $n$ in $(\cos\theta, \sin \theta)$. This can be recast as a set of orthogonality conditions in the following form
\begin{align*}
    \int_{\Sm^1} \int_\Rm s^n e^{\pm ik\theta} \R(s,\theta)\ ds\ d\theta = 0, \quad \forall\ k> n, \quad k-n \text{ even}. 
\end{align*} 
Assuming that said function is supported in the unit disk and changing variables into fan-beam coordinates $(s = \sin\alpha, \theta = \beta+\pi+\alpha)$, the moment conditions are expressed as orthogonality conditions for the data $\D(\beta,\alpha) = \R(\sin\alpha, \beta+\pi+\alpha)$, of the form
\begin{align}
    \int_{\partial_+ SM} \!\!\!\!\!\!\! \cos\alpha\ \sin^n\alpha\ e^{\pm ik (\beta+\alpha)} \D(\beta,\alpha)\ d\alpha\ d\beta = 0, \quad\forall\ n \ge 0,\quad \forall\ k>n, \quad k-n \text{ even}, 
    \label{eq:HL}
\end{align}
while the symmetry condition simply states that $\D\in \V_+$. 

We now prove Theorem \ref{thm:equivalence}, i.e., that this set of orthogonality conditions is strictly equivalent to proving that the frequency content of $\D$ is contained in the range of $P_-$. The proof of Theorem \ref{thm:equivalence} makes use of the following lemma, whose proof (e.g., by induction) is left to the reader:
\begin{lemma} \label{lem:trigo}
    For every natural integer $n\ge 0$, we have
    \begin{align}
	\tspan \{ \cos\alpha \sin^{2k}\alpha,\ 0\le k\le n \} &= \tspan \{ \cos((2k+1)\alpha),\ 0\le k\le n\},  \label{eq:trigo1} \\
	\tspan \{ \cos\alpha \sin^{2k+1}\alpha,\ 0\le k\le n \} &= \tspan \{ \sin(2(k+1)\alpha),\ 0\le k\le n\}.  \label{eq:trigo2}
    \end{align}
\end{lemma}

\begin{proof}[Proof of Theorem \ref{thm:equivalence}] According to \ref{eq:HL}, $\D$ satisfies the moment conditions if and only if, for every $k>0$, 
    \begin{align*}
	\D\ \bot\ \tspan \{ \sin^n \alpha \cos \alpha e^{\pm ik(\alpha+\beta)}, \quad n = 0,\cdots, k-1, \quad k-n \text{ even}\}.
    \end{align*} 
    We now rewrite these $k$-dependent spans in terms of our basis functions $u'_{p,q}$, splitting cases according to the parity of $k$. 
    
    \noindent{\bf Case $k$ even.} Let $k=2p$ for some $p>0$, then the span above equals
    \begin{align*}
	\tspan &\{ \sin^n \alpha \cos \alpha e^{\pm i2p(\alpha+\beta)}, \quad n = 0,\cdots, 2p-1, \quad 2p-n \text{ even}\} \\
	&= \tspan \{ \sin^{2\ell} \alpha \cos \alpha e^{\pm i2p(\alpha+\beta)}, \quad \ell = 0, \cdots, p-1\} \\
	&= \tspan \{ \cos((2\ell+1)\alpha) e^{\pm i2p(\alpha+\beta)}, \quad \ell = 0, \cdots, p-1\} \qquad\qquad (\text{by } \ref{eq:trigo1}) \\
	&= \tspan \{ u'_{2p,l+p}\}_{\ell = 0}^{p-1} \oplus \tspan \{ u'_{-2p,-p-l-1}\}_{\ell = 0}^{p-1}. 
    \end{align*}
    \noindent{\bf Case $k$ odd.} Let $k=2p+1$ for some $p\ge 0$, then the span above equals
    \begin{align*}
	\tspan &\{ \sin^n \alpha \cos \alpha e^{\pm i(2p+1)(\alpha+\beta)}, \quad n = 0,\cdots, 2p, \quad 2p+1-n \text{ even}\} \\
	&= \tspan \{ \sin^{2\ell+1} \alpha \cos \alpha e^{\pm i(2p+1)(\alpha+\beta)}, \quad \ell = 0, \cdots, p-1\} \\
	&= \tspan \{ \sin (2\ell\alpha) e^{\pm i(2p+1)(\alpha+\beta)}, \quad \ell = 1, \cdots, p\} \qquad\qquad (\text{by } \ref{eq:trigo2}) \\
	&= \tspan \{ u'_{2p+1,p+\ell} \}_{\ell=1}^p \oplus \tspan \{ u'_{-2p-1, l-p-1} \}_{\ell=1}^p. 
    \end{align*}
    
    Now, according to Proposition \ref{prop:SVDP}, the orthocomplement of the range of $P_-$ is exactly the $(p,q)$ region spanned by 
    \[ \left\{ u'_{p,q},\quad p\le 2q, \quad q\le 0 \text{ or } p\ge q+\frac{1}{2} \right\}. \]
    The previous two calculations describe this set exactly, $p$-slice by $p$-slice. Theorem \ref{thm:equivalence} is proved. 
\end{proof}

\noindent{\bf On the projection of noisy data onto the range of $I_0$.}
In light of the past two sections, we now propose two approaches in order to project noisy data onto the range of $I_0$. Both approaches are extremely fast as they are based solely on the scattering relation and the fiberwise Hilbert transform (which can be computed slice-by-slice using Fast Fourier Transform), and they allow, by virtue of Theorem \ref{thm:equivalence} to enforce the moment conditions {\em at all orders} at once. 

\begin{enumerate}
    \item For inversion purposes, the reconstruction formula \ref{eq:reconsI0} for functions already contains a projection step. Indeed, applying $A_+^\star H_- A_-$ to the data before backprojection, as prescribed by the formula amounts to applying (up to sign) the adjoint of $P_- = A_-^\star H_- A_+$. By Proposition \ref{prop:SVDP}, applying this operator to the data will remove the content which is supported on the orthocomplement of range $P_-$, while moving data from the space $\V_+$ to $\V_-$, a necessary step before applying backprojection $I_\perp^\sharp$. 
    \item Based on Proposition \ref{prop:SVDC}, we have that $Id + C_-^2$ is exactly the projection onto the range of $P_-$, i.e. the range of $I_0$.  
\end{enumerate}

\subsection{The operator $I_\perp$} \label{sec:Iperp}

We have defined $I_\perp:\dot{H}^{1}(M) \to L^2(\partial_+ SM)$ while the operator ``$I_1$'' defined in \cite{Pestov2004} is the ray transform restricted to one-forms. By virtue of the Helmholtz-Hodge decomposition for a square-integrable one-form $\omega$:
\begin{align*}
  \omega = X f + X_\perp g, \qquad f\in H^1_0(M), \quad g\in \dot{H}^1(M),    
\end{align*}
we have that ``$I_1$''$(\omega) = I(Xf) + I(X_\perp g) = I_\perp g$, so that the range of these transforms is the same, characterized in frequency content by the operator $P_+$. While the reconstruction formula inverting $I_\perp$ in \cite{Pestov2004} only applies to functions in $H^1_0(M)$, we now state by how much the spaces $H^1_0(M)$ and $\dot{H}^1(M)$ differ. 

\begin{lemma} \label{lem:H1}
    The following direct decomposition holds: 
    \begin{align}
	\dot{H}^1(M) = H^1_0 (M) + \dot{H}^1(\ker^0 \eta_+,M) + \dot{H}^1(\ker^0 \eta_-,M).
	\label{eq:H1}
    \end{align}    
\end{lemma}

\begin{proof}[Proof of Lemma \ref{lem:H1}.]
    Let $g\in \dot{H}^1(M)$. Then by trace theorems, $g|_{\partial M}\in H^{\frac{1}{2}}(\partial M)$, and $g|_{\partial M}$, with zero average at the boundary, can be written uniquely as
    \begin{align*}
	g|_{\partial M} (e^{i\theta}) = \sum_{k=1}^\infty a_k e^{ik\theta} + b_k e^{-ik\theta}, \qquad \sum_{k=1}^\infty (|b_k|^2 + |a_k|^2) (1+k) < \infty. 
    \end{align*}
    We now define, inside the domain 
    \begin{align*}
	g_{(-)} (z) := \sum_{k=1}^\infty a_k z^k, \qquad g_{(+)} (z) := \sum_{k=1}^{\infty} b_k \zbar^k.
    \end{align*}
    It is easy to see that $g_{(\pm)} \in \dot{H}^1(\ker^0 \eta_\pm,M)$, and that one may then decompose $g$ uniquely into 
    \begin{align*}
	g = g_{(0)} + g_{(+)} + g_{(-)}, \qquad g_{(0)} := g - g_{(+)} - g_{(-)} \in H^1_0(M).
    \end{align*}
    The lemma is proved.
\end{proof}

We now show that upon applying $I_\perp$ to $g$, the three resulting summands live on orthogonal subspaces of $\V_-$, each of which corresponding to a different nonzero spectral value of $P_+$, and to the $0,\frac{i}{2}$ and $\frac{-i}{2}$-eigenspaces of $C_+$. 

\begin{proposition}\label{prop:Iperpdecomposition}
    The mapping $I_\perp:\dot{H}^1(M)\to \V_-$ maps the direct decomposition \ref{eq:H1} into three orthogonal subspaces of $\V_-$, corresponding to the three distinct nonzero spectral values of $P_+$ (as described in Proposition \ref{prop:SVDP}). Moreover, 
    \begin{align*}
	I_\perp (H_0^1(M)) \subset \ker(C_+), \quad\text{ and }\quad I_\perp (\dot{H}^1(\ker^0 \eta_\pm,M)) \subset \ker (C_+ \mp \frac{i}{2} Id). 
    \end{align*}
\end{proposition}

\begin{proof}[Proof of Proposition \ref{prop:Iperpdecomposition}]
    Write $g\in \dot{H}^1(M)$ as $g = g_{(0)} + g_{(+)} + g_{(-)}$ according to \ref{eq:H1}. We first analyze the summands $g_{(+)}$ and $g_{(-)}$. On to the summand $g_{(-)}$, we have $0 = \eta_- g_{(-)} = (X- iX_\perp)g_{(-)}$, so that, in particular, $X_\perp g_{(-)} = -i X g_{(-)}$, thus $I_\perp g_{(-)} (\beta,\alpha) = -i I[Xg_{(-)}] (\beta,\alpha)$. Using the fundamental theorem of calculus along each line of integration, this equals
  \begin{align}
    I_\perp g_{(-)} (\beta,\alpha) &= -i (g_{(-)} (e^{i(\beta+\pi+2\alpha)}) - g_{(-)} (e^{i\beta})) \nonumber\\
    &= -i \sum_{k=1}^\infty a_k (e^{ik(\beta + \pi + 2\alpha)} - e^{ik\beta}) = -i\pi\sqrt{2} \sum_{k=1}^\infty (-1)^k a_k v_{k,k},
    \label{eq:Iperpgminus}
  \end{align}
  so $I_\perp g_{(-)}$ is in the subspace of $\V_-$ corresponding to the spectral value $-i$ of $P_+$, and according to Proposition \ref{prop:SVDC}, $C_+ I_\perp g_{(-)} = \frac{-i}{2} I_\perp g_{(-)}$. \\
  Similary for $g_{(+)}$, we have $0 = \eta_+ g_{(+)} = (X + iX_\perp)g_{(+)}$, so that, in particular, $X_\perp g_{(+)} = i X g_{(+)}$, thus $I_\perp g_{(+)} (\beta,\alpha) = i I[Xg_{(+)}] (\beta,\alpha)$, which yields
  \begin{align}
    I_\perp g_{(+)} (\beta,\alpha) &= i (g_{(+)} (e^{-i(\beta+\pi+2\alpha)}) - g_{(+)} (e^{-i\beta})) \nonumber \\
    &= i \sum_{k=1}^\infty b_k (e^{-ik(\beta + \pi + 2\alpha)} - e^{-ik\beta}) = -i\pi\sqrt{2} \sum_{k=1}^\infty b_k v_{-k,0}, 
    \label{eq:Iperpgplus}
  \end{align}
  so $I_\perp g_{(+)}$ is in the subspace of $\V_-$ corresponding to the spectral value $i$ of $P_+$, and according to Proposition \ref{prop:SVDC}, $C_+ I_\perp g_{(+)} = \frac{i}{2} I_\perp g_{(+)}$.

  We conclude by showing that the range of $I_\perp$ restricted to $H^1_0(M)$ is orthogonal to both these subspaces. This fact mainly derives from the following lemma, whose proof we relegate at the end of this proof. 
  \begin{lemma}\label{lem:IperpH10}
    For any $f\in H^1_0(M)$, then $I_\perp f$ satisfies: 
    \begin{align*}
      \int_{-\frac{\pi}{2}}^\frac{\pi}{2} I_\perp f (\beta,\alpha) \ d\alpha = 0, \qquad \text{a.e. } \beta\in \Sm^1.
    \end{align*}
  \end{lemma}
  Assuming Lemma \ref{lem:IperpH10} is proved, it is easy to see that, if $f\in H^1_0(M)$, and using the fact that for any $k=1,2,\dots$, $v_{k,k} = -(Id - \SS_A^\star) \phi_{k,0}$ and $\SS_A^\star I_\perp f = -I_\perp f$, we arrive at,
  \begin{align*}
    \dprod{I_\perp f}{v_{k,k}} = 2 \dprod{I_\perp f}{\phi_{k,0}} = \int_{\Sm^1} e^{-ik\beta} \int_{-\frac{\pi}{2}}^\frac{\pi}{2} I_\perp f (\beta,\alpha) \ d\alpha\ d\beta = 0.
  \end{align*}
  Similarly using that $v_{-k,0} = (Id - \SS_A^\star)\phi_{-k,0}$, we arrive at
  \begin{align*}
    \dprod{I_\perp f}{v_{-k,0}} = 2 \dprod{I_\perp f}{\phi_{-k,0}} = \int_{\Sm^1} e^{ik\beta} \int_{-\frac{\pi}{2}}^\frac{\pi}{2} I_\perp f (\beta,\alpha) \ d\alpha\ d\beta = 0.
  \end{align*}
  Since the harmonic content of the ranges of $P_+$ and $I_\perp$ must match, the range of $I_\perp$ restricted to $H_0^1(M)$ cannot but be spanned by the space corresponding to the spectral value $-2i$ of $P_+$, on which $C_+$ vanishes identically. Proposition \ref{prop:Iperpdecomposition} is proved.
\end{proof}

\begin{proof}[Proof of Lemma \ref{lem:IperpH10}] It is enough to prove this for $f\in \C^\infty_c (M)$ and the result follows by density. We compute
    \begin{align*}
	\int_{-\frac{\pi}{2}}^\frac{\pi}{2} I_\perp f (\beta,\alpha) \ d\alpha &= \int_{-\frac{\pi}{2}}^\frac{\pi}{2} \int_0^{2\cos\alpha} X_\perp f (e^{i\beta} + t e^{i(\beta + \alpha + \pi)})\ dt\ d\alpha \\
	&= \int_{-\frac{\pi}{2}}^\frac{\pi}{2} \frac{\partial}{\partial\alpha} \left( \int_0^{2\cos\alpha} \frac{-1}{t} f (e^{i\beta} + t e^{i(\beta + \alpha + \pi)})\ dt \right)\ d\alpha,
    \end{align*}
    upon using the chain rule and the fact that $f|_{\partial M} =0$. Now since $f$ has compact support inside $M$, for $\alpha$ close enough to $\pm \frac{\pi}{2}$, $f$ is identically zero along the segment $\{e^{i\beta} + t e^{i (\beta + \alpha+\pi)}, \quad t\in (0,2\cos\alpha) \}$, so the last right-hand-side above is zero. Lemma \ref{lem:IperpH10} is proved.     
\end{proof}

\subsection{The ranges over $\ker^m \eta_\pm$} \label{sec:otherranges}

Let us first show how the characterization of the range of $I_0$ helps us understand the lack of injectivity over tensors. Fix $m \in\Zm$ and $f\in L^2(M)$ and suppose we want to compute $I[f(\x)e^{im\theta}] = I_m f$. Then we obtain: 
\begin{align}
   \begin{split}
       I [f(\x) e^{im\theta}](\beta,\alpha) &= \int_0^{2\cos\alpha} f(e^{i\beta} + te^{i(\beta+\pi+\alpha)}) e^{im(\beta+\pi+\alpha)}\ dt \\
       &= (-1)^m e^{im(\beta+\alpha)} I_0 f (\beta,\alpha). 
    \end{split}
    \label{eq:relation}
\end{align}
This implies that, for instance if $m=2$, the range of $I_2$ amounts to translating the range of $I_0$ in the $(p,q)$-plane by $(2,1)$, so that the overlap between the ranges of $I_0$ and $I_2$ is rather large. The next observation is that, in fact, the only part of the range of $I_2$ which does not lie in that of $I_0$ is precisely $I(\ker^2 \eta_-)$ ! 

We now formulate the range characterizations of $I(\ker^m \eta_\pm)$. To this end, if $\{u_k\}_{k=0}^\infty$ is a sequence of orthogonal vectors in a Hilbert space such that there exists $C\ge 1$ with $C^{-1} \le \|u_k\| \le C$ for every $k$, we define: 
\begin{align}
    h^{\frac{1}{2}} \left(\{u_k\}_{k=0}^\infty\right) := \left\{ v = \sum_{k=0}^\infty a_k \widehat{u_k}, \qquad \sum_{k=0}^\infty (k+1) |a_k|^2 <\infty \right\}.
    \label{eq:spandef}
\end{align}
\begin{proposition}\label{prop:rangesker} For any $m\in \Zm$, we have
    \begin{align}
                  I(\ker^{2m} \eta_-) &= h^{\frac{1}{2}} \left(\{u'_{2m+k,m+k}\}_{k=0}^\infty\right), \label{eq:rangeeven1} \\
                   I(\ker^{2m} \eta_+) &= h^{\frac{1}{2}} \left(\{u'_{2m-k,m}\}_{k=0}^\infty\right), \label{eq:rangeeven2} \\
                  I(\ker^{2m+1} \eta_-) &= h^{\frac{1}{2}} \left(\{v_{2m+1+k,m+k+1}\}_{k=0}^\infty\right), \label{eq:rangeodd1} \\
                  I(\ker^{2m+1} \eta_+) &= h^{\frac{1}{2}} \left(\{v_{2m+1-k,m+1}\}_{k=0}^\infty\right). \label{eq:rangeodd2}
    \end{align}
\end{proposition}

\begin{proof}[Proof of Proposition \ref{prop:rangesker}] {\bf Study of $\ker^0 \eta_\pm$.} Suppose $f\in L^2(M)$ is a holomorphic function, i.e. $f\in \ker^0 \eta_- = \ker \dbar$. Then we may write $f$ as $f(z) = \sum_{k=0}^\infty a_k z^k$ for some numbers $\{a_k\}_k$ such that 
    \begin{align*}
	\sum_{k=0}^\infty \frac{|a_k|^2}{k+1} = \frac{1}{\pi} \|f\|^2_{L^2(M)} < \infty. 
    \end{align*}
    Then for any $(\beta,\alpha)\in \partial_+ SM$, 
    \begin{align}
	I_0 f(\beta,\alpha) &= \int_0^{2\cos\alpha} f(e^{i\beta}(1-te^{i\alpha}))\ dt \nonumber \\
	&= \sum_{k=0}^\infty (-1)^k a_k e^{ik(\beta+\alpha)}  \int_0^{2\cos\alpha} (t-e^{-i\alpha})^k\ dt \nonumber \\
	&= \sum_{k=0}^\infty \frac{(-1)^k a_k}{k+1} e^{ik\beta} (e^{i(2k+1)\alpha}+ (-1)^ke^{-i\alpha}) \nonumber \\
	&= \pi\sqrt{2} \sum_{k=0}^\infty \frac{(-1)^k a_k}{k+1} u'_{k,k}. \label{eq:I0holo}
    \end{align}
    
    This implies straightforwardly that 
    \begin{align*}
	I(\ker^0 \eta_-) = \left\{ \sum_{k=0}^\infty b_k u'_{k,k}, \qquad \sum_{k=0}^\infty (k+1) |b_k|^2 < \infty \right\} = h^{\frac{1}{2}} \left(\{u'_{k,k}\}_{k=0}^\infty\right).
    \end{align*}
    
    If $f\in L^2(M)$ is antiholomorphic, of the form $f(z) = \sum_{k=0}^\infty a_k \zbar^k$, then $\overline{f}$ is holomorphic and we can compute, using \ref{eq:I0holo},
    \begin{align*}
	I_0 f = \overline{I_0 \bar{f}} = \overline{\pi\sqrt{2} \sum_{k=0}^\infty \frac{(-1)^k \overline{a_k}}{k+1} u'_{k,k}} = \pi\sqrt{2} \sum_{k=0}^\infty \frac{(-1)^k a_k}{k+1} \overline{u'_{k,k}} = \pi\sqrt{2} \sum_{k=0}^\infty \frac{a_k}{k+1} u'_{-k,0},
    \end{align*}
    where we have used \ref{eq:conj} in the last equality. We therefore deduce that
    \begin{align*}
	I(\ker^0 \eta_+) = \left\{ \sum_{k=0}^\infty b_k u'_{-k,0}, \qquad \sum_{k=0}^\infty (k+1) |b_k|^2 < \infty \right\} = h^{\frac{1}{2}}\left(\{u'_{-k,0}\}_{k=0}^\infty\right).
    \end{align*}

    {\bf Study of $\ker^m \eta_\pm$ for $m\ne 0$.} From \ref{eq:relation}, recall that 
    \begin{align*}
	I[f(\x)e^{im\theta}] (\beta,\alpha) = (-1)^m e^{im(\beta+\alpha)} I_0 f (\beta,\alpha), 
    \end{align*} 
    and $f(\x) e^{im \theta}\in \ker^{m}\eta_\pm$ if and only if $f\in \ker^0 \eta_\pm$ so that multiplication by $e^{im(\beta+\alpha)}$ is an isomorphism between $I(\ker^0 \eta_\pm)$ and $I(\ker^{m}\eta_\pm)$. In terms of our bases $\B$ and $\B'$, this is framed as follows:
    \begin{description}
	\item[Even case:] we have that $I[f(\x) e^{2im \theta}] = \phi_{2m,m} I_0 f$, and using the fact that $\phi_{2m,m} u'_{p,q} = u'_{p+2m, q+m}$ for any $p,q,m$, we deduce \ref{eq:rangeeven1} and \ref{eq:rangeeven2}. 
	\item[Odd case:] we have that $I[f(\x) e^{i(2m+1) \theta}] = - \phi'_{2m+1,m} I_0 f$, and using the fact that $\phi'_{2m+1,m} u'_{p,q} = v_{2m+1+p,m+q+1}$ for any $p,q,m$, we deduce \ref{eq:rangeodd1} and \ref{eq:rangeodd2}. 
    \end{description} 
    The proof of Proposition \ref{prop:rangesker} is complete. 
\end{proof}

\section{Reconstruction formulas - Proof of Theorem \ref{thm:recons}} \label{sec:recons}

In light of Theorems \ref{thm:decomp} and \ref{thm:range}, every $m$-tensor $f$ has an equivalent $m$-tensor $g$ such that $If = Ig$ and such that the ray transforms of each component of $g$ live on orthogonal subspaces of $L^2(\partial_+ SM)$. We now explain how to reconstruct each component. Section \ref{sec:I0Iperp} covers the inversion of $I_0$ and $I_\perp$, while Section \ref{sec:other} covers the reconstruction of elements in any space $\ker^{m} \eta_\pm$ for $m\in \Zm$. 

\subsection{Inversion of $I_0$ and $I_\perp$} \label{sec:I0Iperp}

Inversion formulas for $I_0$ were long known in the parallel geometry (see \cite{Radon1917,Natterer2001}), and reconstruction formulas in fan-beam geometry were obtained by changing variable in the former inversion \cite{Herman1976}. Seen in a context of transport equations on Riemannian surfaces, the first inversion formulas for the operators $I_0$ and $I_\perp$ (restricted to smooth functions with compact support) appeared in \cite{Pestov2004}, though not being represented using the $A_+^\star H A_-$ operator as we presently do. As explained earlier, adding this operator allows to project data onto the ranges of $I_0$ and $I_\perp$. For a proof of the inversion formulas justifying the use of the operator $A_+^\star H A_-$, see \cite[Proposition 2.2]{Monard2015} (note that the error operators $W,W^\star$ appearing there account for curvature and vanish identically in the present Euclidean setting). Above, we have defined the backprojection operators as follows: for $\D (\beta,\alpha)$ defined on $\partial_+ SM$ and $\x\in M$, 
\begin{align}
  \begin{split}
    I_0^\sharp g(\x) &:= \frac{1}{2\pi} \int_{\Sm^1} \D (\theta+\sin^{-1} (\x\cdot\btheta^\perp), - \sin^{-1} (\x\cdot\btheta^\perp))\ d\theta, \qquad \left(\btheta^\perp := \binom{-\sin\theta}{\cos\theta}\right) \\
    I_\perp^\sharp g(\x) &:= \frac{1}{2\pi} \nabla\cdot \left( \int_{\Sm^1} \btheta^\perp\ \D (\theta+\sin^{-1} (\x\cdot\btheta^\perp), - \sin^{-1} (\x\cdot\btheta^\perp))\ d\theta \right).    
  \end{split}
  \label{eq:backproj}
\end{align}

\begin{remark}
  Although we favored the unweighted space $L^2(\partial_+ SM)$ for reasons of range description here, it is to be noted that when deriving reconstruction formulas, the operators $I_0^\sharp$ and $I_\perp^\sharp$ appear naturally and they are the adjoints of $I_0$ and $I_\perp$ when considering the space $L^2(\partial_+ SM, \cos\alpha)$ as their codomain, instead of the unweighted one.   
\end{remark}

\noindent{\bf Inversion of $I_\perp$ over $\dot{H}^1(M)$.} As the previous section only covered the case of function vanishing at the boundary, we must now consider the inversion over the decomposition $g_{(0)} + g_{(+)} + g_{(-)}$ as in Section \ref{sec:Iperp}. 

It turns out that applying the inversion formula above to $I_\perp (g_{(0)} + g_{(+)} + g_{(-)})$ does not only pick up $g_{(0)}$ but also a {\em fraction} of $g_{(+)}$ and $g_{(-)}$. In order to remove this coupling, we derive independent reconstruction formulas for $g_{(\pm)}$ exploiting their analyticity, and we introduce an additional boundary operator which will remove $I_\perp (g_{(+)} + g_-)$ from $I_\perp g$ to reconstruct $g_{(0)}$ separately.

\begin{proof}[Proof of Equations \ref{eq:reconsgminus} and \ref{eq:reconsgplus}]
    Call $\D =  I_\perp (g_{(0)} + g_{(+)} + g_{(-)})$ and let us reconstruct the functions $g_{(\pm)}$ from $\D$. Using the calculation \ref{eq:Iperpgminus}, we have
  \begin{align*}
      g_{(-)}(z) = \sum_{k=1}^\infty a_k z^k &= \frac{i}{\pi\sqrt{2}} \sum_{k=1}^\infty (-1)^k  z^k \frac{\dprod{\D}{v_{k,k}}}{\dprod{v_{k,k}}{v_{k,k}}}\\
    & = \frac{i}{4\pi^2} \int_{\partial_+ SM} \D(\beta,\alpha) \sum_{k=1}^\infty\left( (-z e^{-i(\beta+2\alpha)})^k - (ze^{-i\beta})^k \right) \ d\alpha\ d\beta \\
    &= \frac{i}{4\pi^2} \int_{\partial_+ SM} \D(\beta,\alpha) \left( \frac{-z e^{-i(\beta+2\alpha)}}{1+z e^{-i(\beta+2\alpha)}} - \frac{ze^{-i\beta}}{1-ze^{-i\beta}} \right) \ d\alpha\ d\beta \\
    &= \frac{1}{2i\pi^2} \int_{\partial_+ SM} \D(\beta,\alpha) \frac{Id - \SS_A^\star}{2} \frac{ze^{-i\beta}}{1-ze^{-i\beta}}\ d\alpha\ d\beta.
  \end{align*}
  If the data is consistent in the sense that $\frac{Id - \SS^\star_A}{2}\D = \D$, then we can simplify this formula into \ref{eq:reconsgminus}, hence the formula. 
  
  On to the function $g_{(+)}$, using the fact that $I_\perp \overline{g_{(+)}} = \overline{I_\perp g_{(+)}}$ and the fact that $\overline{g_{(+)}}$ is holomorphic, we can use the previous formula to arrive at \ref{eq:reconsgplus}. 
\end{proof}

\begin{proof}[Proof of Equation \ref{eq:reconsgzero}] We suppose the following formula (see e.g. \cite[Proposition 2.2]{Monard2015}) to hold for functions on $H^1_0(M)$: 
    \begin{align*}
	g_{(0)} = \frac{-1}{8\pi} I_0^\sharp A_+^\star H A_- I_\perp g_{(0)}. 
    \end{align*}
    Calling $\D = I_\perp (g_{(0)} + g_{(+)} + g_{(-)})$, all we need to prove is that $(Id + (A_-^\star H A_{-})^2) \D = I_\perp g_{(0)}$, or in other words, that 
    \begin{align*}
	\left(A_-^\star H A_- \right)^2 \D = (2C_+)^2 \D = - I_\perp (g_{(+)} + g_{(-)}), 
    \end{align*}
    where $C_+$ is defined in Section \ref{sec:scatrel}. This is an immediate consequence of Proposition \ref{prop:Iperpdecomposition}. 
\end{proof}

\subsection{Reconstruction of elements of $\ker^m \eta_\pm$ for $m\in \Zm$} \label{sec:other}
The present section provides justifications for equations \ref{eq:reconsholo1}, \ref{eq:reconsholo2}, \ref{eq:reconsholo3}  and \ref{eq:reconsholo4} in Theorem \ref{thm:recons}.

\paragraph{The case $m=0$ of holomorphic and antiholomorphic functions:}

From the calculation \ref{eq:I0holo}, we can derive a reconstruction formula for $f$: since the $u'_{p,q}$'s are orthogonal, equation \ref{eq:I0holo} implies that the reconstruction of each coeffient is given by
\[ a_k = \frac{(-1)^k (k+1)}{\pi\sqrt{2}} \frac{\dprod{I_0 f}{u'_{k,k}}}{\dprod{u'_{k,k}}{u'_{k,k}}} = \frac{(-1)^k (k+1)}{2\pi\sqrt{2}} \dprod{I_0 f}{u'_{k,k}}, \qquad k\ge 0.   \]
We can also derive an integral reconstruction for $f$ thanks to the following computation, valid for $|z|<1$,
\begin{align*}
    f(z) &= \sum_{k\ge 0} a_k z^k \\
    &= \frac{1}{2\pi\sqrt{2}} \sum_{k\ge 0}  (-1)^k (k+1) \dprod{I_0 f}{u'_{k,k}} z^k \\
    &= \frac{1}{4\pi^2} \left\langle I_0 f,\ \sum_{k\ge 0} (-1)^k (k+1) z^k (e^{-i\alpha} e^{-ik(\beta+2\alpha)} + (-1)^k e^{-ik\beta} e^{i\alpha})\right\rangle \\
    &= \frac{1}{2\pi^2} \int_{\partial_+ SM} I_0 f(\beta,\alpha) \frac{1}{2}\left( \frac{e^{-i\alpha}}{(1+ze^{-i(\beta+2\alpha)})^2} + \frac{e^{i\alpha}}{(1-ze^{-i\beta})^2}\right)\ d\beta\ d\alpha,
\end{align*}
where we have used the power series $(1-\zeta)^{-2} = \sum_{k\ge 0} (k+1)\zeta^k$, true for any $|\zeta|<1$. Upon defining $G_0(z;\beta,\alpha) := \frac{e^{i\alpha}}{(1-ze^{-i\beta})^2}$, the formula can be summarized as
\begin{align*}
  f(z) = \frac{1}{2\pi^2} \int_{\frac{-\pi}{2}}^{\frac{\pi}{2}}\int_{\Sm^1} I_0 f(\beta,\alpha) \frac{1}{2}\left( G_0(z;\beta,\alpha) + G_0(z;\SS_A(\beta,\alpha)) \right)\ d\beta\ d\alpha.
\end{align*}
By symmetry, the half-sum inside the integral can be replaced by either term in the sum. On the other hand, the half-sum can be useful when reconstructing the function and projecting data on the appropriate range at the same time. If one only keeps the second term for instance, the inversion becomes extremely fast to implement: 
\begin{align}
    f(z) = \frac{1}{2\pi^2}\int_{\Sm^1} \frac{1}{(1-ze^{-i\beta})^2} \int_{-\frac{\pi}{2}}^{\frac{\pi}{2}} I_0 f(\beta,\alpha) e^{i\alpha} \ d\alpha\ d\beta, \qquad f\in \ker^0 \eta_-.
    \label{eq:holoInversion}
\end{align}

Now if $f$ is antiholomorphic, then $\bar f$ is holomorphic and since $I_0 \bar f = \overline{I_0 f}$, we can deduce a reconstruction procedure for an antiholomorphic $f$ as well: applying \ref{eq:holoInversion} to $\bar f$ from $\overline{I_0 f}$ and taking complex conjugates. This yields
\begin{align}
    f(z) = \frac{1}{2\pi^2}\int_{\Sm^1} \frac{1}{(1-\bar{z} e^{i\beta})^2} \int_{-\frac{\pi}{2}}^{\frac{\pi}{2}} I_0 f(\beta,\alpha) e^{-i\alpha} \ d\alpha\ d\beta, \qquad f\in \ker^0 \eta_+.
    \label{eq:antiholoInversion}
\end{align}

\paragraph{Reconstruction of elements in $\ker^m \eta_\pm$ for $m\ne 0$.}
Let $h\in \ker^m \eta_-$, of the form $h(\x,\theta) = f(\x) e^{im\theta}$ with $\dbar f = 0$. Using \ref{eq:relation}, we can deduce immediately that 
\[ I_0 f (\beta,\alpha) = (-1)^m e^{-im(\beta+\alpha)} I[f(\x) e^{im\theta}] (\beta,\alpha). \]
Combining this with \ref{eq:holoInversion}, we arrive at
\begin{align*}
    f(z) &= \frac{1}{2\pi^2}\int_{\Sm^1} \frac{1}{(1-ze^{-i\beta})^2} \int_{-\frac{\pi}{2}}^{\frac{\pi}{2}} I_0 f(\beta,\alpha) e^{i\alpha} \ d\alpha\ d\beta \\
    &= \frac{(-1)^m}{2\pi^2}\int_{\Sm^1} \frac{e^{-im\beta}}{(1-ze^{-i\beta})^2} \int_{-\frac{\pi}{2}}^{\frac{\pi}{2}} I[f e^{im\theta}](\beta,\alpha) e^{i(1-m)\alpha} \ d\alpha\ d\beta,
\end{align*}
which is what we had to prove. 

If $h\in \ker^m \eta_+$, the proof of the reconstruction formula is a similar combination of \ref{eq:relation} and \ref{eq:antiholoInversion}. 

\section{On the possibility of other decompositions}\label{sec:decomps}

Let us restrict the present discussion to the case of even tensors, though the case of odd tensors is similar. The decomposition described in Theorem \ref{thm:decomp} is based on a particular choice of ``main harmonic'' $g_0$. However, for a $2n$-tensor $f$ with data $If$, one could very well pick any integer $-n\le k\le n$ and construct a $2n$-tensor $g$ with $If = Ig$ such that, 
\begin{itemize}
    \item For $-n\le \ell < k$, $g_{2\ell}\ e^{i2\ell\theta} \in \ker^{2\ell} \eta_+$, reconstructible from $If$ using the formulas derived in the present article. 
    \item For $k<\ell \le n$, $g_{2\ell}\ e^{i2\ell\theta} \in \ker^{2\ell} \eta_-$, reconstructible from $If$ using the formulas derived in the present article. 
    \item The ``main harmonic'' is of the form $e^{2ik\theta} g_{2k} (\x)$ with $g_{2k} \in L^2(M)$, and is to be reconstructed from the transform $I [e^{2ik\theta} g_{2k}]$, for which the author has given inversion formulas in \cite[Theorem 5.2]{Monard2013a}. These formulas are similar in spirit to the case $k=0$ except that they involve conjugations of the Hilbert transform by factors $e^{\pm i2k\theta}$. 
\end{itemize}

Such decompositions are also continous in the sense that $\|g\|_{L^2(SM)} \le C \|f\|_{L^2(SM)}$ and enjoy the same efficiency in implementation. The main difference with the case $k=0$ is that the harmonic $g_{2k}$ would now be the one which contains all visible singularities. Figure \ref{fig:decomps} illustrates, on a 2-tensor reconstruction problem, two different ways in which to view the frequency content of the transform of a 2-tensor, to reconstruct two different candidates. 

\begin{figure}[htpb]
    \centering
    \includegraphics[width=0.49\textwidth]{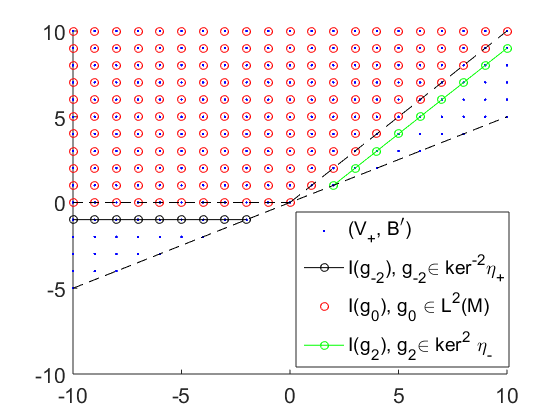}
    \includegraphics[width=0.49\textwidth]{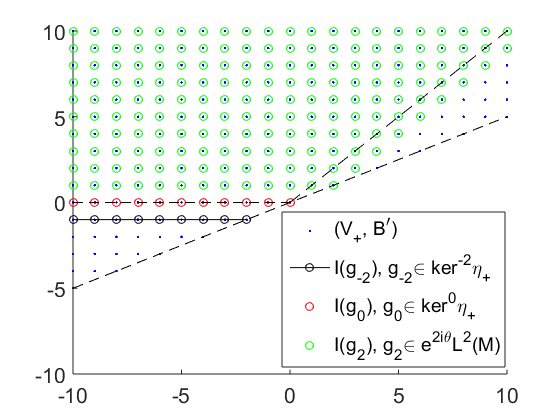}    
    \caption{Two possible ways to view the frequency content of the transform of a $2$-tensor, leading to two different candidates $g = g_{-2}\ e^{-i2\theta} + g_0 + g_2\ e^{i2\theta}$. }
    \label{fig:decomps}
\end{figure}

\section{Numerical experiments} \label{sec:numerics}

We now show some numerical illustrations of the reconstruction formulas, one for a tensor of even order, one for a tensor of odd order, using the {\tt Matlab} code previously documented by the author in \cite{Monard2013} in the context of Riemannian surfaces. The underlying grid is $300\times 300$ cartesian, the discretization of the data space $\Sm^1\times (-\frac{\pi}{2},\frac{\pi}{2})$ is $600\times 300$ equispaced. All computations terminate within seconds on a regular personal computer.

\subsection{Experiment 1: Reconstruction of a second order tensor}
We first illustrate Theorem \ref{thm:recons} with an example of a second-order tensor $f = f_0 + f_2 e^{2i\theta} + f_{-2} e^{-2i\theta}$. For simplicity of display, we make $f$ real-valued by imposing the constraints $\overline{f_0} = f_0$ and $\overline{f_{2}} = f_{-2}$. In this case, if we write $f_2 = f_{2}^r + i f_2^i$, the tensor $f$ takes the form
\begin{align*}
  f = f_0 + 2 f_2^r \cos(2\theta) - 2 f_2^i \sin(2\theta).
\end{align*}
In tensor notation, this corresponds to the symmetric 2-tensor
\begin{align}
  f = (f_0 + 2 f_2^r)\ dx \otimes dx - 4 f_2^i\ \sigma(dx \otimes dy) + (f_0 - 2 f_2^r)\ dy \otimes dy.
  \label{eq:fex1}
\end{align}
The functions $f_0$, $f_2^r$, $f_2^i$ are given in Fig. \ref{fig:f}. As prescribed in Theorem \ref{thm:recons}, we reconstruct the equivalent tensor $g = g_0 + g_2 + g_{-2}$ according to the formulas there. One should note in particular that since the data is real-valued, the reconstructed $g$ satisfies $\overline{g_0} = g_0$ and $\overline{g_2} = g_{-2}$, so that we may represent $g$ just like $f$, that is, in terms of functions $g_0$, $g_2^r$ and $g_2^i$ as in \ref{eq:fex1}. These three functions are represented in Fig. \ref{fig:g}.

The forward data $If$, as well as the pointwise difference $|If-Ig|$ are given on Fig. \ref{fig:errorI}.

\begin{figure}[htpb]
  \centering
  \includegraphics[trim = 20 0 20 60, clip, width=.3\textwidth]{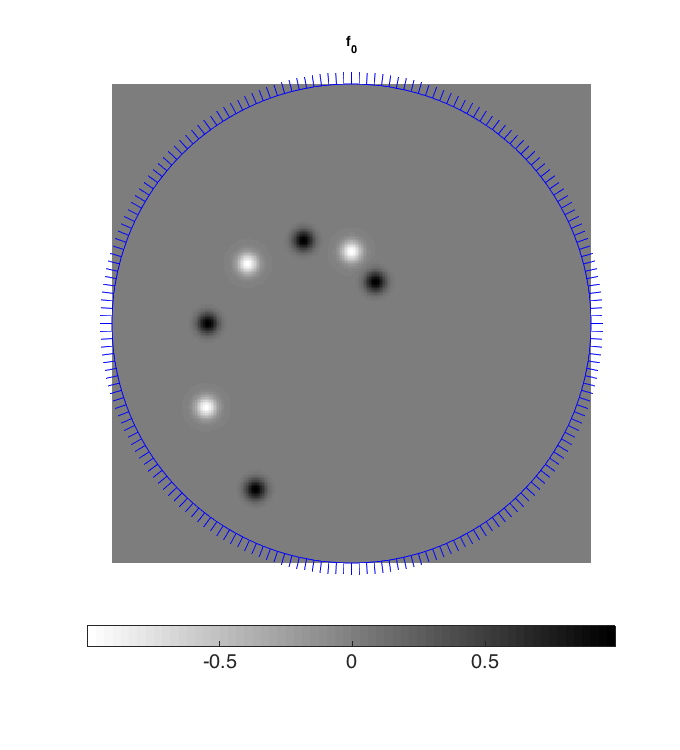}
  \includegraphics[trim = 20 0 20 60, clip, width=.3\textwidth]{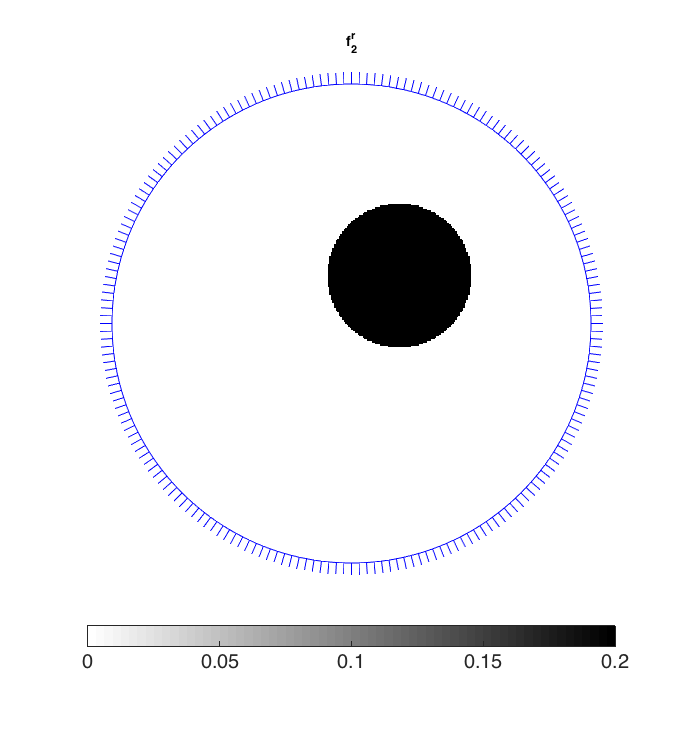}
  \includegraphics[trim = 20 0 20 60, clip, width=.3\textwidth]{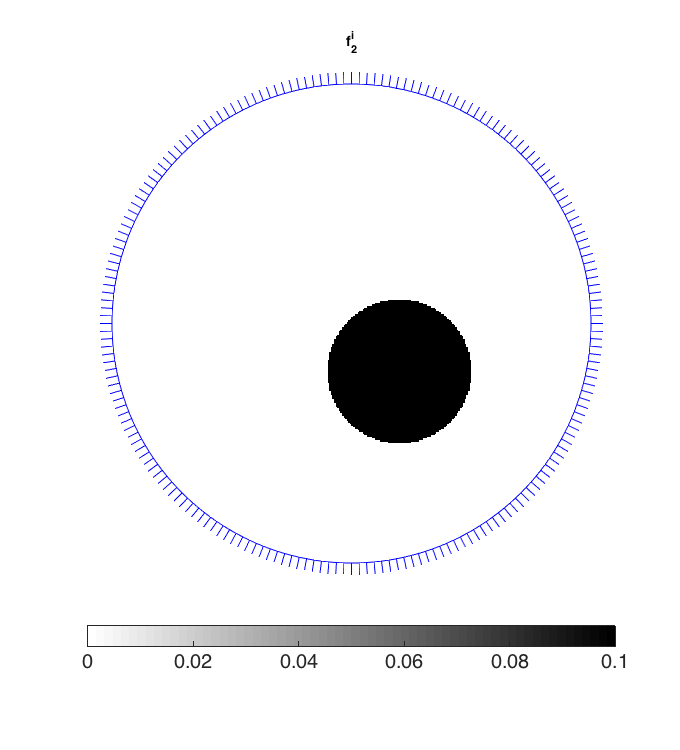}
  \caption{Exp. 1: Second order tensor $f$ as defined in \ref{eq:fex1} via three real-valued functions $f_0$ (left), $f_2^r$ (middle) and $f_2^i$ (right).}
  \label{fig:f}
\end{figure}

\begin{figure}[htpb]
  \centering
  \includegraphics[trim = 20 0 20 60, clip, width=.3\textwidth]{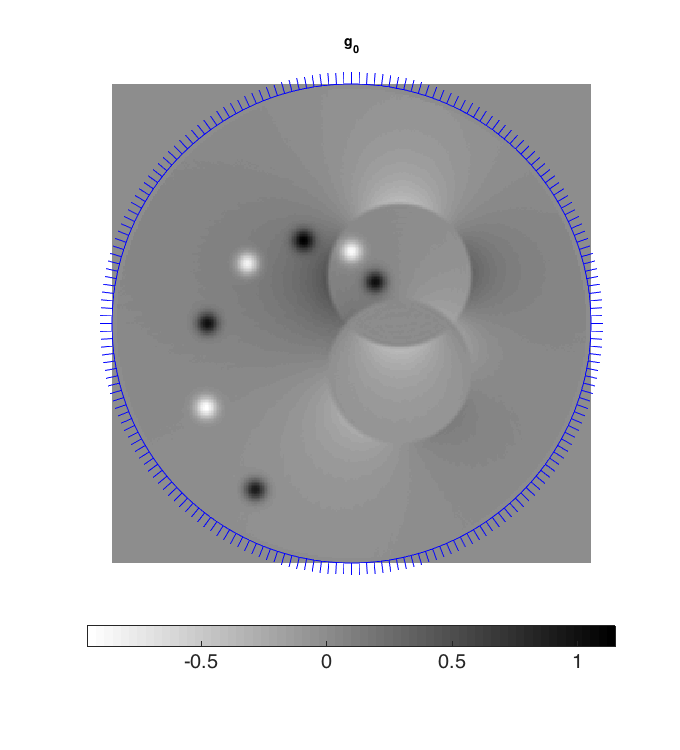}
  \includegraphics[trim = 20 0 20 60, clip, width=.3\textwidth]{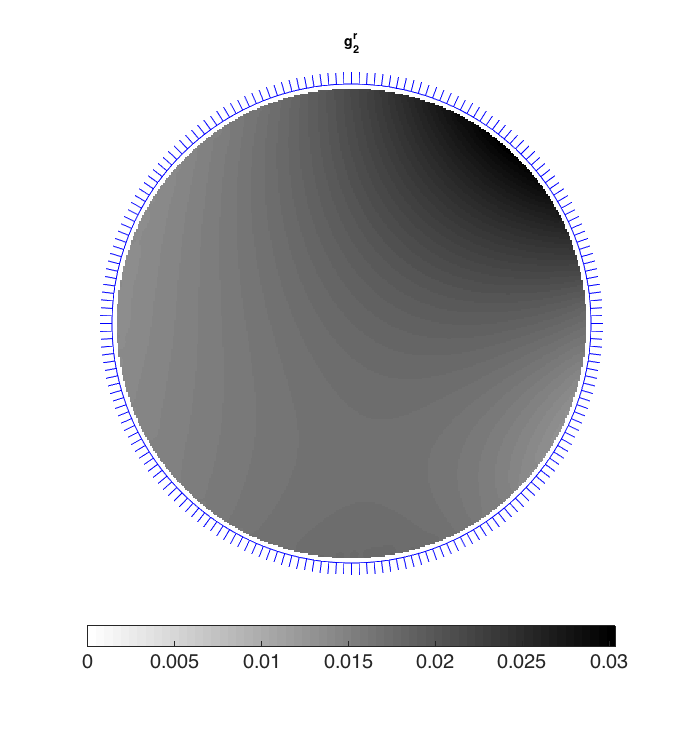}
  \includegraphics[trim = 20 0 20 60, clip, width=.3\textwidth]{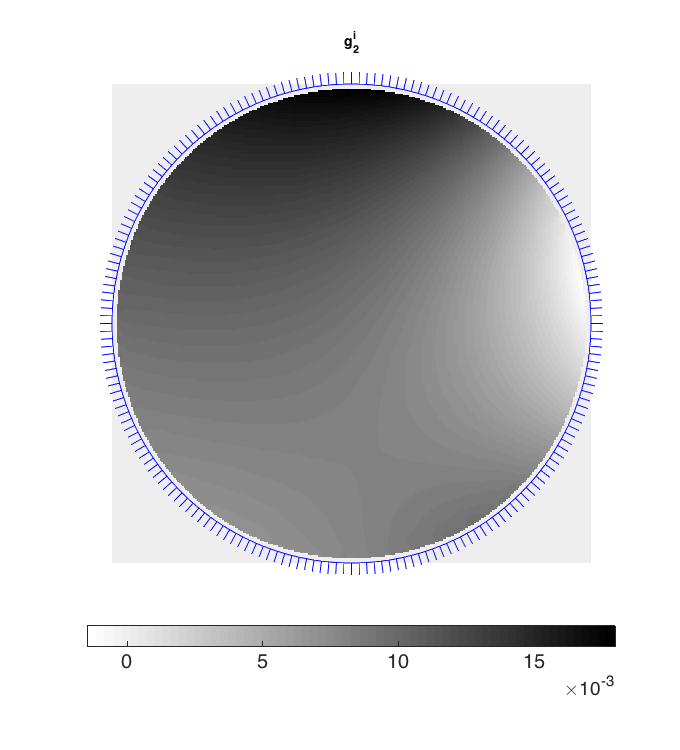}
  \caption{Exp. 1: Second order tensor $g$ reconstructed from the data $If$, given by three real-valued functions $g_0$ (left), $g_2^r$ (middle) and $g_2^i$ (right).}
  \label{fig:g}
\end{figure}

\begin{figure}[htpb]
  \centering
  \includegraphics[trim = 20 0 20 60, clip, width=.45\textwidth]{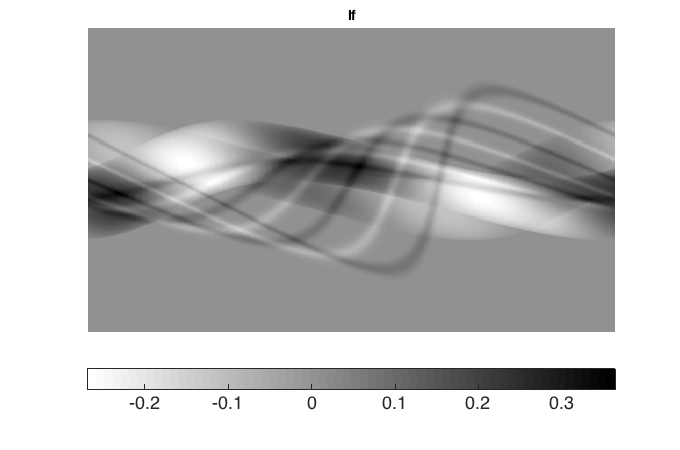}
  \includegraphics[trim = 20 0 20 60, clip, width=.45\textwidth]{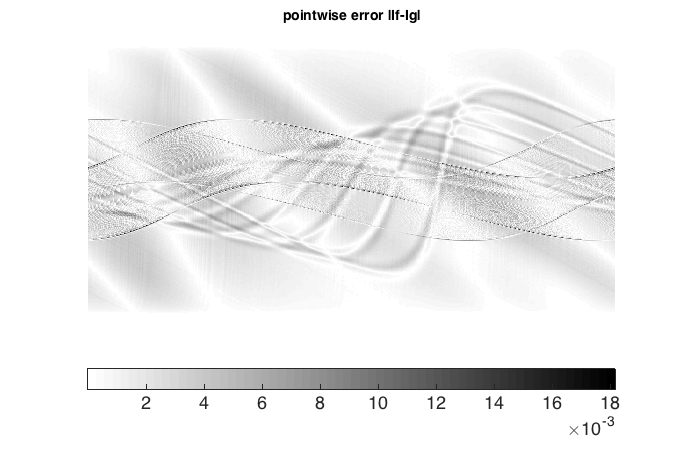}  
  \caption{Exp. 1: Left: ray transform $If$. Right: pointwise difference $|If-Ig|$, where $g$ is the tensor reconstructed from data $If$, displayed on Fig. \ref{fig:g}.}
  \label{fig:errorI}
\end{figure}

Some comments are in order: 
\begin{enumerate}
  \item Even if the tensor has compact support (which is the case here), the reconstructed tensor may not have compact support.
  \item The inability to separate the singularities of $f_0, f_2, f_{-2}$ in the reconstruction is not a weakness of the method: the lack of injectivity implies the loss of such an information. As seen in section \ref{sec:decomps}, one could choose to reconstruct another $h = h_{-2} + h_0 + h_2$ with $h_{-2} \in \ker^{-2} \eta_+$, $h_0 \in \ker^0 \eta_+$ and $h_2 \in $ containing all reconstructed singularities, and this would make another perfectly acceptable candidate. 
\end{enumerate}

\subsection{Experiment 2: Reconstruction of a solenoidal vector field}

In the case of tensors of odd order, there is an additional numerical technicality: as the Cauchy integral type of inversion formulas may become unstable too close to the boundary, we need to cut off the reconstruction near the boundary. This in turn would create artifacts when recomputing $I_\perp$ of the reconstructed function and comparing it to the initial data, which makes this comparison method unreliable. 

As an alternative, we first show that the inversion procedure of $I_\perp$ over $\dot{H}^1(M)$ is successful almost up to the boundary by comparing directly the reconstruction with the initial function (we can do that because this problem is injective). Once this is done, the next section will give an example of a reconstruction on an example of a third-order tensor. 

In this example we take $f = f_{(0)} + f_\partial$ as the sum of a compactly supported function $f_{(0)}$ (sum of peaked gaussians) and a non-compactly supported harmonic term $f_\partial = \Re(z^3)$, as in Figure \ref{fig:solvf1} (left). The forward data $I_\perp f$ is on the right of Figure \ref{fig:solvf1}. We first apply $Id + (A_-^\star H A_-)^2$ to the data to extract $I_\perp f_{(0)}$ (Fig. \ref{fig:solvf2}, left) and reconstruct $f_{(0)}$ (see Fig. \ref{fig:solvf3}, left) from it. For convenience, $I_\perp f_\partial$, extracted from the data, is visualized Fig. \ref{fig:solvf2} (right). We then apply formulas \ref{eq:reconsgminus}-\ref{eq:reconsgplus} to the data to reconstruct $f_\partial$ (see Fig. \ref{fig:solvf3}, middle). The pointwise error inside the centered disk of radius $0.99$ is given on Fig. \ref{fig:solvf3} (right). 

\begin{figure}[htpb]
    \centering
    \includegraphics[trim = 20 0 20 50, clip, width=.3\textwidth]{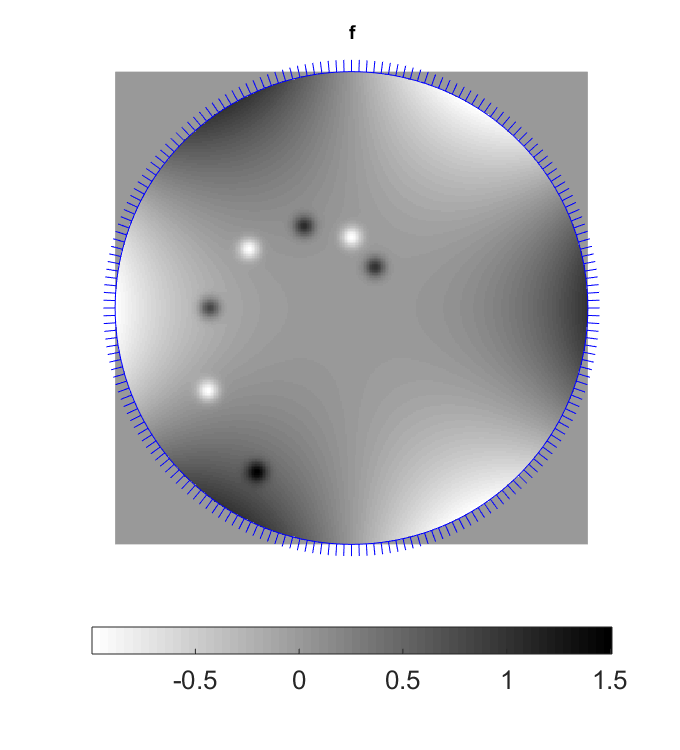}
    \includegraphics[trim = 20 0 20 20, clip, width=.45\textwidth]{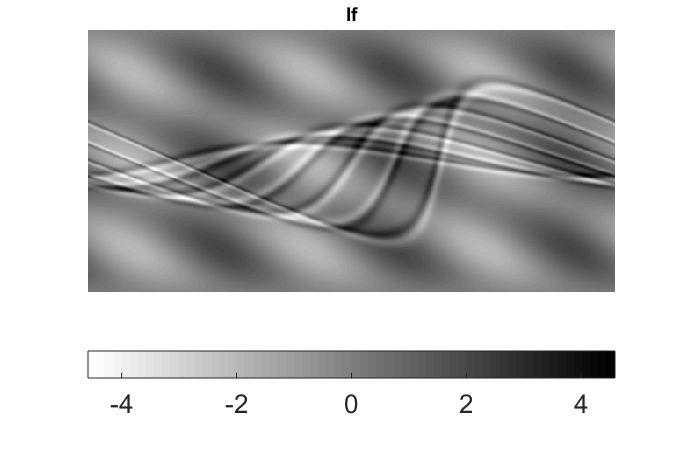}
    \caption{Exp. 2: Left: function $f$ generating the solenoidal vector field $X_\perp f$. Right: the data $I_\perp f$.}
    \label{fig:solvf1}
\end{figure}

\begin{figure}[htpb]
    \centering
    \includegraphics[trim = 20 0 20 25, clip, width=.45\textwidth]{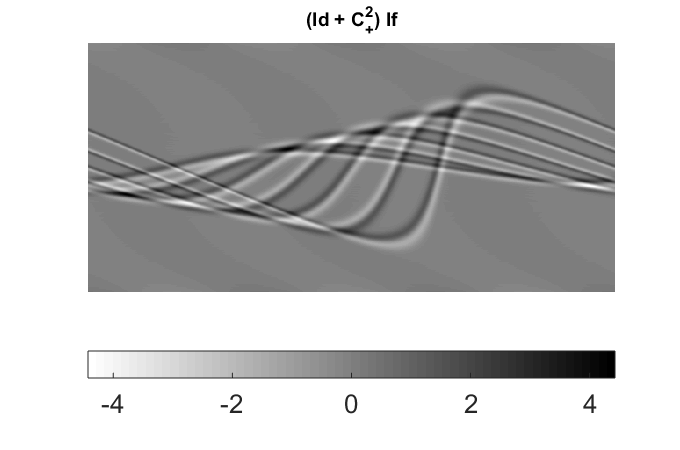}
    \includegraphics[trim = 20 0 20 25, clip, width=.45\textwidth]{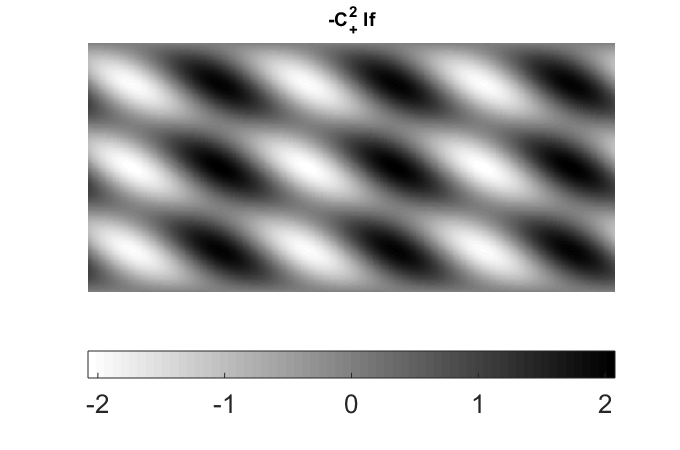}    
    \caption{Exp. 2: Decomposition of $I_\perp (f_{(0)}+f_\partial)$ into $I_\perp(f_{(0)})$ (left) via the operator $Id + (A_-^\star H A_-)^2$, and $I_\perp (f_\partial)$ (right) via the operator $-(A_-^\star H A_-)^2$. The sum of both gives back the initial data $I_\perp f$ displayed on Fig. \ref{fig:solvf1} (right).}
    \label{fig:solvf2}
\end{figure}

\begin{figure}[htpb]
    \centering
    \includegraphics[trim = 20 0 20 50, clip, width=.3\textwidth]{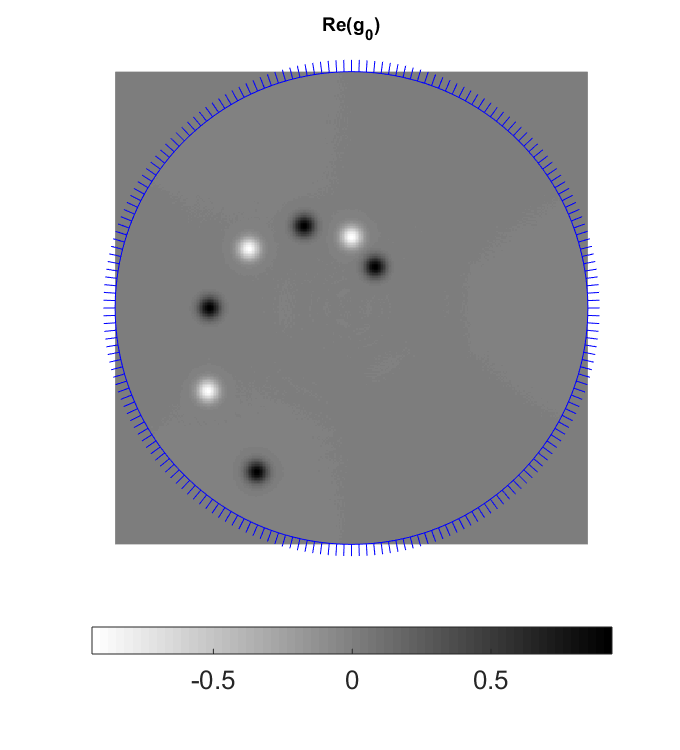}
    \includegraphics[trim = 20 0 20 50, clip, width=.3\textwidth]{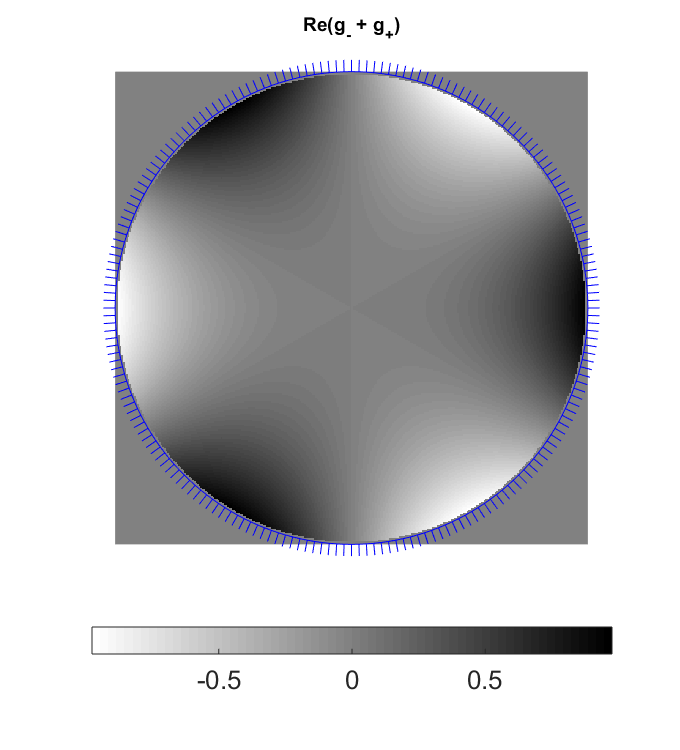}
    \includegraphics[trim = 20 0 20 50, clip, width=.3\textwidth]{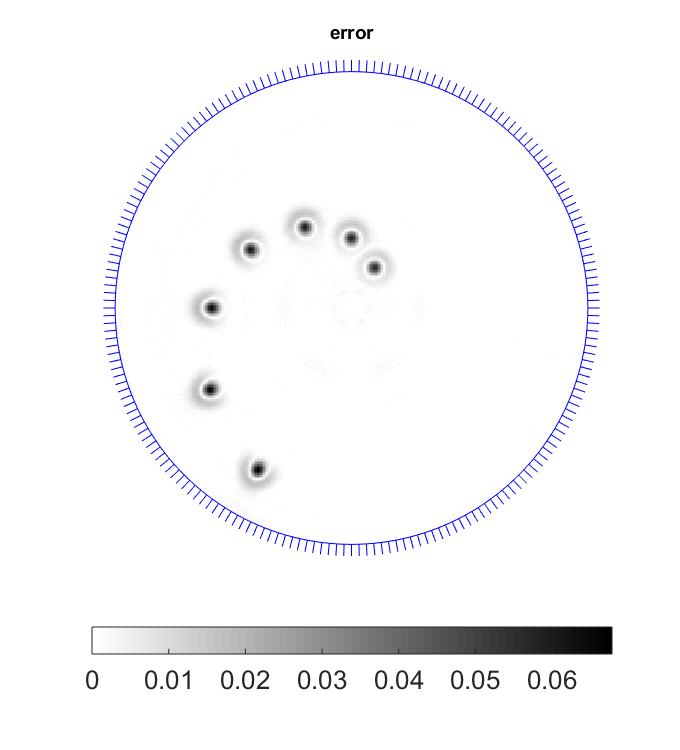}
    \caption{Exp. 2: Left to right: reconstructed $f_{(0)}$, reconstructed $f_\partial$, pointwise error on $f$. }
    \label{fig:solvf3}
\end{figure}

\subsection{Experiment 3: Reconstruction of a third-order tensor}
We now give an example of a third-order tensor $f = f_{\pm 1} e^{\pm i \theta} + f_{\pm 3} e^{\pm i3\theta}$, where again, to simplify the exposition, we assume $f$ real-valued so that $f_{-3} = \overline{f_3}$ and $f_{-1} = \overline{f_1}$. If we write $f_k = f_k^r + if_k^i$ for $k=1,3$, we obtain that the tensor $f$ takes the form
\begin{align*}
    \frac{f}{2} &= f_1^r \cos \theta - f_1^i \sin\theta + f_3^r \cos(3\theta) - f_3^i \sin (3\theta) \\
    &= (f_1^r + f_3^r)\cos^3 \theta - (f_1^i+3f_3^i) \cos^2 \theta\sin\theta  \\
    &\qquad + (f_1^r - 3f_3^r) \cos\theta\sin^2\theta + (-f_1^i+f_3^i) \sin^3\theta, 
\end{align*}
which in tensor notation corresponds to the tensor
\begin{align*}
    \frac{f}{2} &= (f_1^r + f_3^r)\ dx^{\otimes 3} - (f_1^i+3f_3^i)\ \sigma(dx^{\otimes 2} \otimes dy) \\
                       &\qquad + (f_1^r - 3f_3^r)\ \sigma(dx\otimes dy^{\otimes 2}) + (-f_1^i+f_3^i)\ dy^{\otimes 3}.
\end{align*}
An example of such functions is given Figure \ref{fig:Exp3_1}, with the ray transform of the corresponding $3$-tensor given in Figure \ref{fig:Exp3_2}. By Theorem \ref{thm:decomp}, we reconstruct an equivalent real-valued $3$-tensor of the form 
\begin{align*}
    g = X_\perp g_0 + 2 g_3^r \cos(3\theta) - 2 g_3^i \sin (3\theta), \quad g_0 \in \dot{H}^1(M), \quad (g_3^r\pm ig_3^i)e^{\pm i3\theta} \in \ker^{\pm 3} \eta_{\mp}, 
\end{align*}
according to Theorem \ref{thm:recons}. The reconstructed functions are given in Figure \ref{fig:Exp3_3}. Note that in this case, the function $g_0$ contains all visible singularities, though all of them smoother by $1$ degree than the ones of the initial tensor $f$. This is because the contribution of $g_0$ in the equivalent tensor is in fact $X_\perp g_0$ and not $g_0$ itself. 

\begin{figure}[htpb]
    \centering
    \includegraphics[trim = 20 0 20 50, clip, width=.24\textwidth]{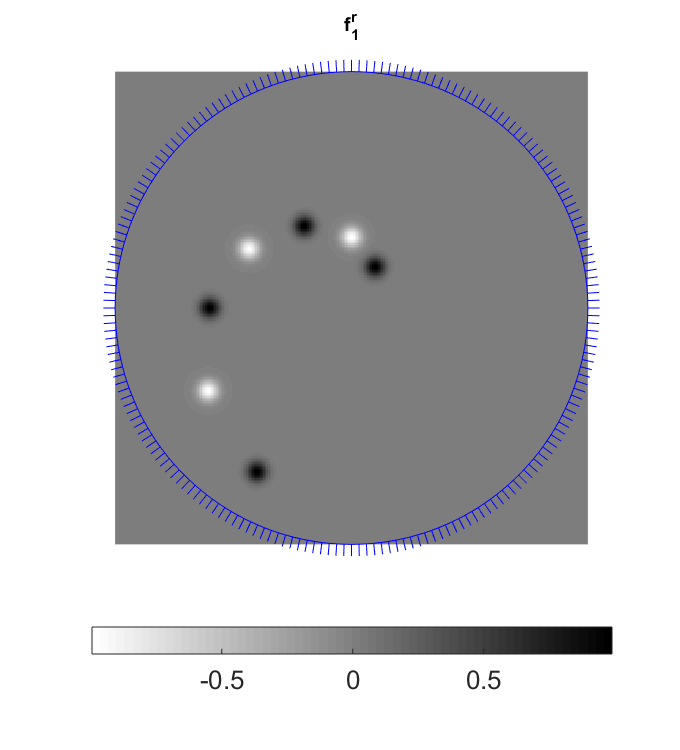}
    \includegraphics[trim = 20 0 20 50, clip, width=.24\textwidth]{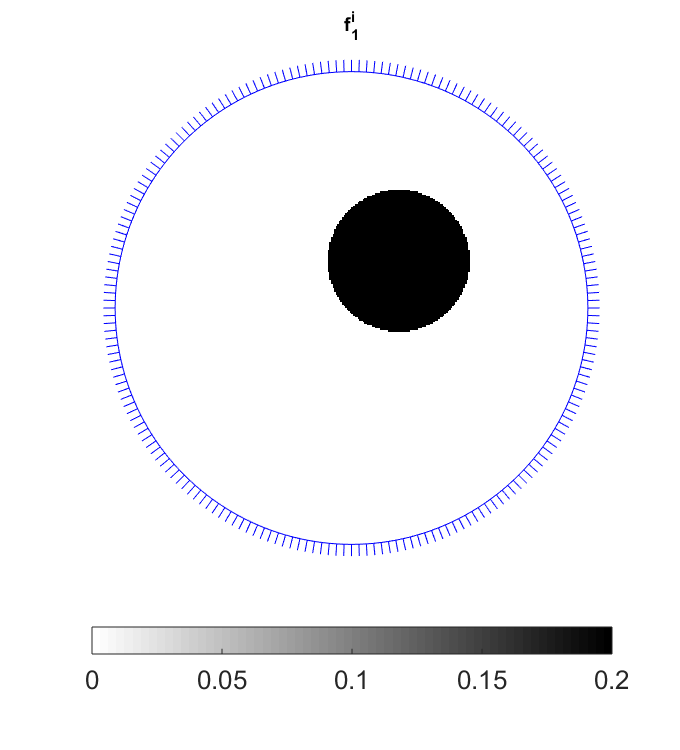}
    \includegraphics[trim = 20 0 20 50, clip, width=.24\textwidth]{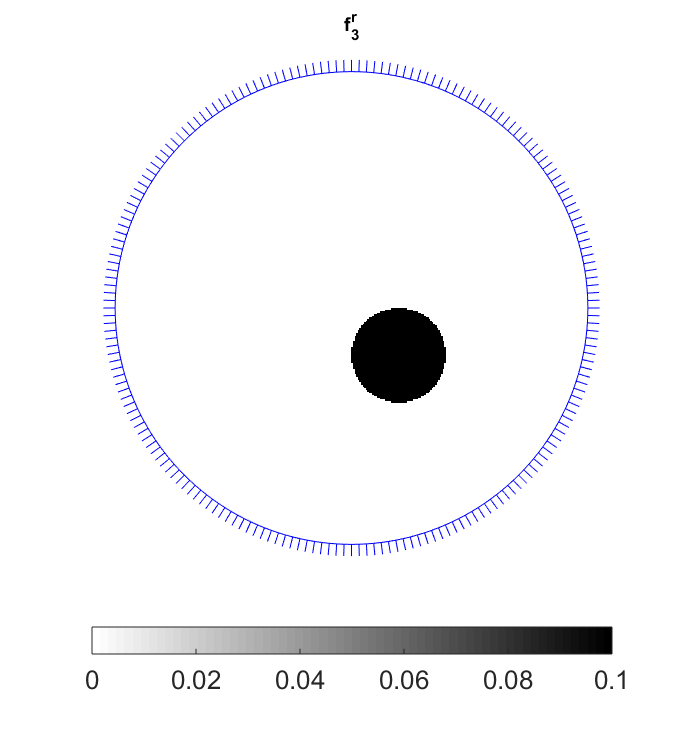}
    \includegraphics[trim = 20 0 20 50, clip, width=.24\textwidth]{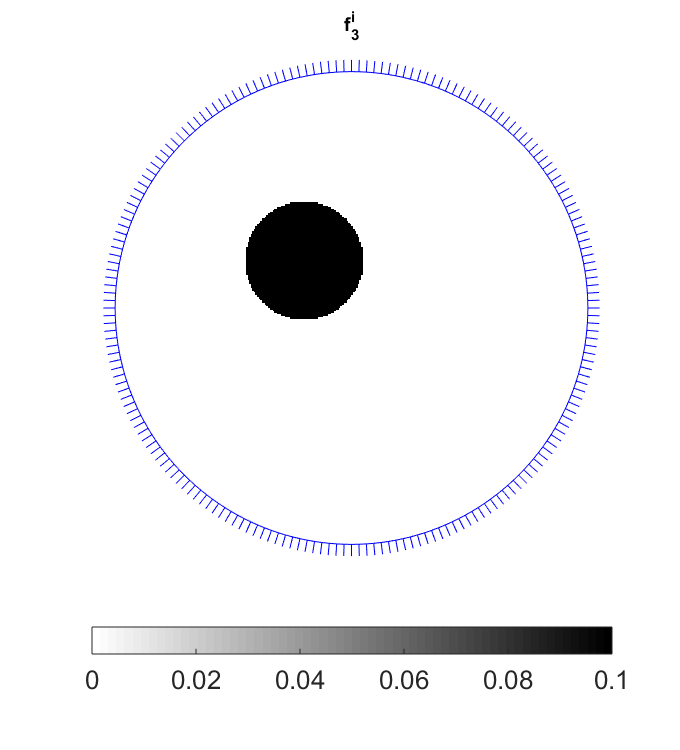}
    \caption{Exp. 3: Left to right: the functions $f_1^r$, $f_1^i$, $f_3^r$ and $f_3^i$ characterizing a real-valued third-order tensor.}
    \label{fig:Exp3_1}
\end{figure}

\begin{figure}[htpb]
    \centering
    \includegraphics[trim = 20 0 20 20, clip, width=.45\textwidth]{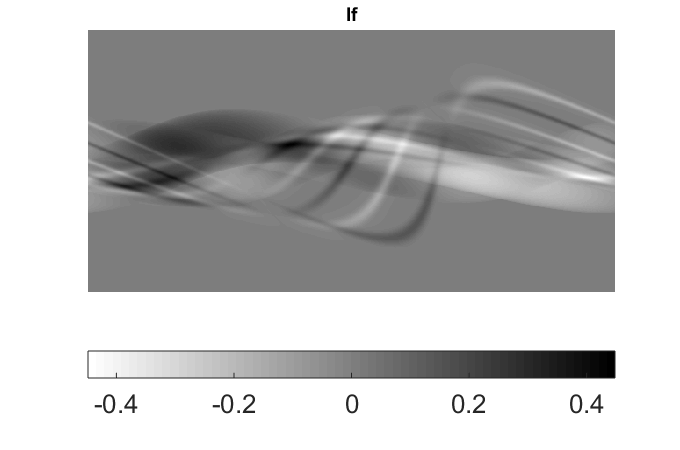}
    \caption{Exp. 3: the data $If$ of the third-order tensor defined by the functions in Fig. \ref{fig:Exp3_1}.}
    \label{fig:Exp3_2}
\end{figure}

\begin{figure}[htpb]
    \centering
    \includegraphics[trim = 20 0 20 50, clip, width=.3\textwidth]{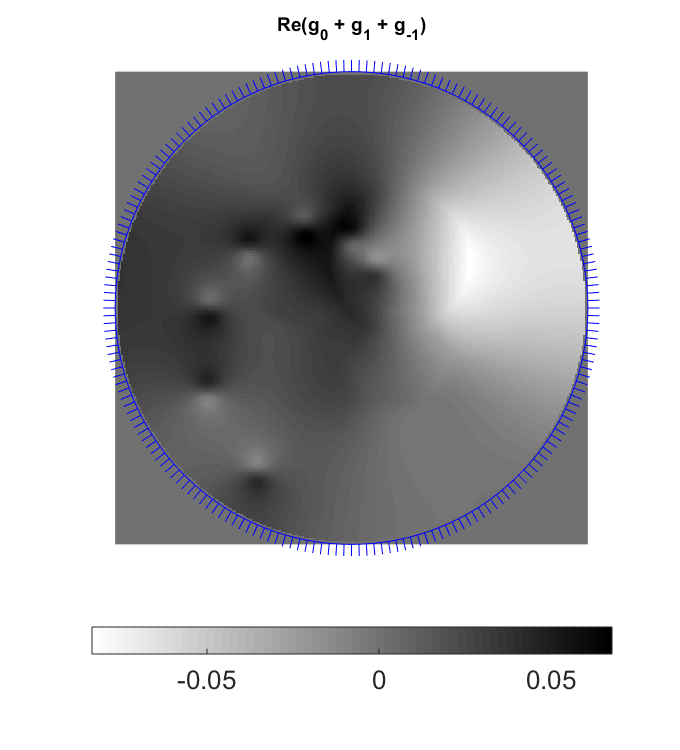}
    \includegraphics[trim = 20 0 20 50, clip, width=.3\textwidth]{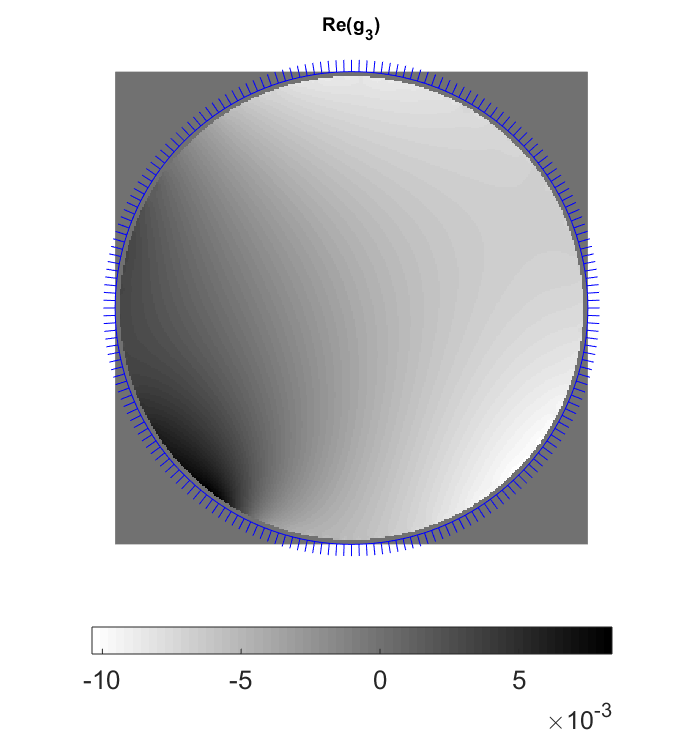}
    \includegraphics[trim = 20 0 20 50, clip, width=.3\textwidth]{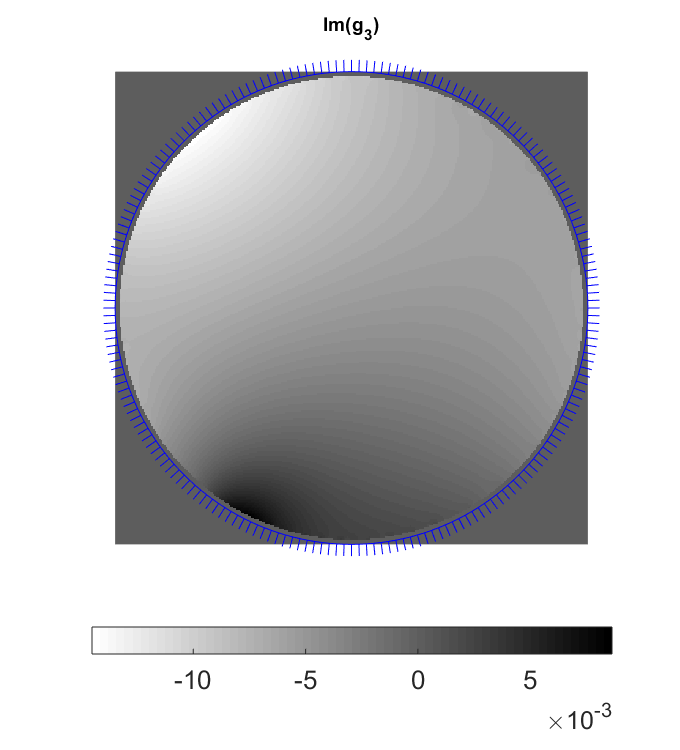}
    \caption{Exp. 3: Left to right: the functions $g_0$, $g_3^r$, $g_3^i$ of the reconstructed tensor $g = X_\perp g_0 + 2g_3^r \cos(3\theta) - 2g_3^i \sin (3\theta)$, whose ray transform equals $If$.}
    \label{fig:Exp3_3}
\end{figure}

\end{document}